\documentclass{article}

\usepackage[top=1 in, bottom=1 in, left=1 in, right = 1 in]{geometry}
\usepackage{amsmath}
\usepackage{hyperref}
\usepackage{cleveref}
\usepackage{autonum}
\usepackage{amsthm}
\usepackage{amssymb}
\usepackage{mathrsfs}  
\usepackage{setspace}
\usepackage{graphicx}
\usepackage{amsbsy}
\usepackage{xcolor}
\usepackage{algorithm}
\usepackage{algorithmic}

\allowdisplaybreaks
\usepackage{caption}
\usepackage{comment}
\newcommand{\pnorm}[3]{\| #1 \|_{L^{#2}(#3)}}

\newcommand{\abs}[1]{\left| #1  \right|}

\newcommand{\heat}[1]{(\partial_t - d_{#1}\Delta) #1}

\newcommand{\B}[1]{{\boldsymbol #1}}

\newcommand{\omu}{\overline{\mu}}

\newcommand{\cM}{\mathcal S}

\def\XXint#1#2#3{{\setbox0=\hbox{$#1{#2#3}{\int}$ }
		\vcenter{\hbox{$#2#3$ }}\kern-.6\wd0}}

\newtheorem{theorem}{Theorem}[section]
\newtheorem{lem}{Lemma}[section]
\newtheorem{cor}{Corollary}[section]
\newtheorem{definition}[theorem]{Definition}
\newtheorem{proposition}[theorem]{Proposition}
\newtheorem{remark}[theorem]{Remark}
\newtheorem{ass}[theorem]{Assumption}

\usepackage{todonotes}

\newcommand{\MB}[1]{\todo[author=MB,inline,color=green!30,size=\footnotesize]{\textsf{#1}}}
\newcommand{\CB}[1]{{\color{black}#1}}

\title{Optimal Control of a Reaction-Diffusion Epidemic Model with Noncompliance}

\author{Marcelo Bongarti\thanks{Weierstrass Institute for Applied Analysis and Stochastics, Berlin, Germany
  (bongarti@wias-berlin.de).} \and Christian Parkinson\thanks{Department of Mathematics and Department of Computational Mathematics, Science and Engineering, Michigan State University, East Lansing, MI, USA
  (chparkin@msu.edu).}
\and Weinan Wang \thanks{Department of Mathematics, University of Oklahoma, Norman, OK, USA
  (ww@ou.edu).}}

\date{\today}

\usepackage{amsopn}

\begin{document}
\maketitle

\begin{abstract}
In this paper, we consider an optimal distributed control problem for a reaction-diffusion-based SIR epidemic model with human behavioral effects. We develop a model wherein non-pharmaceutical intervention methods are implemented, but a portion of the population does not comply with them, and this noncompliance affects the spread of the disease. Drawing from social contagion theory, our model allows for the spread of noncompliance parallel to the spread of the disease. The quantities of interest for control are the reduction in infection rate among the compliant population, the rate of spread of noncompliance, and the rate at which non-compliant individuals become compliant after, e.g., receiving more or better information about the underlying disease. We prove the existence of global-in-time solutions for fixed controls and study the regularity properties of the resulting control-to-state map. The existence of optimal control is then established in an abstract framework for a fairly general class of objective functions. Necessary first--order optimality conditions are obtained via a Lagrangian based stationarity system. We conclude with a discussion regarding minimization of the size of infected and non-compliant populations and present simulations with various parameters values to demonstrate the behavior of the model.
\end{abstract}

\section{Introduction}

With the arrival of COVID-19, most nations implemented non-pharmaceutical interventions (NPIs) to prevent disease spread including mandated mask-wearing, stay-at-home orders, and social distancing. While there is solid evidence that such measures are effective in slowing the spread of the disease \cite{SD1,SD2}, public health experts quickly realized that the effectiveness was hampered by noncompliance with NPIs among a nontrivial portion of the population, and that this noncompliance had a nontrivial effect on the disease spread \cite{NC1,NC3,NC4,NC2}. Thus it is important to develop mathematical models for epidemiology which incorporate these sort of human behavioral effects. 

We model noncompliance with NPIs in a manner borrowed from social contagion theory \cite{SC1,SC2}, which is the idea that human behaviors and emotions can ``spread" much like a disease. Among other applications, social contagion theory has been successfully used to model alcohol and drug use \cite{SC3,SC4}, spread of deleterious mental health conditions \cite{SC5}, participation in violent and/or gang-related activity \cite{SC7,SC6}, and teenage sexual behavior \cite{SC8}. Inspired by social contagion theory, \cite{bongarti2022alternative,parkinson2023analysis} consider a compartmental Susceptible-Infectious-Recovered (SIR) type models wherein, parallel to the disease spread, noncompliance with NPIs spreads through the population and affects the dynamics of the disease. Compartmental modeling of epidemics using partial differential equations (PDE) to account for spatially heterogeneous disease spread is now commonplace \cite{RD1,RD4,RD2,RD5,RD3}, though most such manuscripts do not incorporate human behavior. One notable exception is the work of Berestycki et al. \cite{beres2,beres1} where a compartmental ordinary differential equation (ODE) model is generalized by allowing the susceptible population and transmission rate to depend heterogeneously on a newly introduced variable which describes the level of risk acceptance or aversion throughout the population. In doing so, they derive a reaction-diffusion system where the populations are spatially homogeneous, but diffusion occurs with respect to risk aversion variable. In this way, the authors of \cite{beres2,beres1} are also modeling the simultaneous spread of a disease and human behavior (which has a bearing on the spread of the disease), though in a manner which is different from that of \cite{bongarti2022alternative,parkinson2023analysis}. \CB{The authors of \cite{Gumel1, Gumel2} also consider noncompliance with NPIs, though their models are spatially homogeneous and involve simple linear transfer between the compliant and noncompliant populations, in contrast to the spatial heterogeneity and mass-action transfer between the two considered in \cite{bongarti2022alternative,parkinson2023analysis} and in this manuscript. Besides these, there are many other interesting extensions of the basic SIR model which employ ODE and PDE---for example, multi-group models which stratify the population by age, co-morbities or a variety of other factors \cite{MG1,MG2,MG3,MG4,MG5}, and hybrid discrete-continuous models which pair ODE or PDE with an agent-based approach \cite{HybridModel1,HybridModel2,HybridModel3}. However, as stated above, few have incorporated human behavioral concerns in PDE models for epidemiology. }  

In addition to incorporating spatial heterogeneity and noncompliance with NPIs into our model, we are interested in control of an epidemic by some hypothetical governing agency. That is, we would like to model lawmakers' decision making regarding the institution and enforcement of NPIs. Optimal control of PDEs is now a well-established field \cite{OptControlPDE2,Hinze,OptControlPDE1}, and has found specific application in biological models \cite{Garvie1,Garvie2,Lenhart}. Several authors apply optimal control to various aspects of PDE models for epidemiology \cite{Chang1,Chang2,Jang,Zhou,Zhu}, but again, to the authors' knowledge, none of these incorporate human behavioral effects. 

In this manuscript, we present and analyze a reaction-diffusion model for disease spread akin to that of \cite{parkinson2023analysis}. However, we also include control maps which allow the governing agency to deter the spread of disease and noncompliance. We prove existence of local optimal controls for our system and explore the behavior of our model through simulations in various parameter regimes. The remainder of this manuscript is organized as follows. In the ensuing subsection, we present and discuss the system of PDEs that we will analyze. In section \ref{sec:globalExistence}, we discuss global existence for the system when the control maps are fixed. In section \ref{sec:control}, we address optimal control of our system, prove existence and uniqueness of optimal control maps, and characterize the optimal control maps. In section \ref{sec:simulation}, we present the results of simulating our model and discuss practical implications. Finally, in \ref{sec:conclusion}, we include a brief summary of our work and possible avenues for future research. 


\subsection{Our Model} \label{modeling}

As stated above, we consider the model proposed by the second and third authors in \cite{parkinson2023analysis}. For completeness of exposition, we discuss some of the modeling decisions briefly here. 

A standard assumption for compartmental models of epidemiology is that the disease spreads according to the principle of mass action. We assume that noncompliance with NPIs spreads in an analogous way: any time a compliant individual interacts with a noncompliant individual, they have some chance to become noncompliant. In the model, $S,I,R$ will represent the \emph{compliant} susceptible, infected, and recovered individuals, respectively. Throughout the manuscript, noncompliance is denoted with an asterisk, so that $S^*,I^*,R^*$ represent the \emph{noncompliant} susceptible, infectious, and recovered populations, respectively. We also introduce $N^* = S^*+I^*+R^*$ to represent the total noncompliant population. We assume a spatially heterogeneous ``birth" rate $b(x)$ which is the only manner in which new members are introduced into the population, and we assume a constant ``death" rate $\delta$ which is the only manner in which members are removed from the population. All newly introduced members are susceptible, but they are split into the portion $\xi \in [0,1]$ who are \emph{compliant} susceptible and the portion $1-\xi$ who are \emph{noncompliant} susceptible. Following standard SIR notation, we let $\beta$ denote the constant transmission rate, and $\gamma$ denote the constant recovery rate for the disease. Because noncompliance is spreading via mass action as well, we let $\overline \mu$ denote the transmission rate of noncompliance. We assume that among the compliant population, there is a reduction in infectivity $\underline \alpha \in [0,1]$ due to compliance with NPIs. Accordingly, in any mass action terms representing disease spread, $S$ and $I$ will multiplied by $1-\underline \alpha$, while noncompliant individuals do not receive this reduction in infectivity. All this leads to the reaction-diffusion system: \begin{equation}\begin{aligned}\label{PDE_SYS_pre}
    \heat S &= \xi b(x) -\beta (1-\underline \alpha)SI_M -\omu SN^* - \delta S, \\
	\heat I &= \beta (1-\underline \alpha)SI_M - \gamma I - \omu IN^*  - \delta I, \\
	\heat R &= \gamma I - \omu RN^* - \delta R,\\
	\heat {S^*} &=(1-\xi)b(x) -\beta S^*I_M + \omu SN^* - \delta S^*,\\
	\heat {I^*} &= \beta S^*I_M - \gamma I^* +\omu IN^* - \delta I^*, \\
	\heat {R^*}&= \gamma I^* + \omu RN^* - \delta R^*.
\end{aligned}\end{equation} \CB{Here, for shorthand, we have introduced $I_M = (1-\underline \alpha)I + I^*$ to denote the \emph{actively mixing infectious population}.} These are the infectious individuals who are contributing to disease spread. 

Finally, we incorporate policy-maker controls in three ways. \begin{itemize} 
\item[($i$)] To model increasing use of NPIs, we assume that the policy-maker can achieve an improved reduction in infectivity among the compliant population by choosing $\alpha(x,t) \in [\underline \alpha,1]$. Thus, we will instead multiply $S$ and $I$ by $1-\alpha(x,t)$ in any mass action terms corresponding to disease spread. Choosing $\alpha(\cdot,\cdot) \equiv \underline \alpha$ would correspond to the uncontrolled, ``baseline" case, whereas choosing $\alpha(\cdot,\cdot)\equiv 1$ would correspond to the maximally controlled case, which will entirely halt disease spread among the compliant populations.
    \item[($ii$)] To model strategies such as public service announcements or educational campaigns aimed at deterring noncompliant behavior, the policy-maker can achieve a reduction in the infectivity of noncompliance given by $\mu(x,t) \in [0,\omu]$. Having chosen the values $\mu(x,t)$, the infectivity of noncompliance will be $\omu - \mu(x,t)$. Here the uncontrolled case is $\mu(\cdot,\cdot)\equiv 0$, where there is no reduction in infectivity of noncompliance, whereas the maximally controlled case is $\mu(\cdot,\cdot) \equiv \overline \mu$ whereupon spread of noncompliance is deterred as much as possible. 
    \item[($iii$)] To model intervention strategies such as educational programs aimed specifically at the noncompliant population, we assume that the policy-maker can introduce a ``recovery" rate for noncompliance $\nu(x,t) \in [0,\overline \nu]$. Here $\overline \nu$ is the maximal achievable recovery rate from noncompliance. While the introduction of ``recovery" from noncompliance may seem artificial, social scientists have observed that extreme opinions often die out or regress to mean over time \cite{RM}, so there may be precedent to include some nonzero baseline rate of recovery as in \cite{parkinson2023analysis} (this observation also lends credence to the decision of \cite{beres2,beres1} to model risk acceptance/aversion using diffusion which will have an averaging effect).  
\end{itemize} The maximally controlled scenario will be most effective in slowing disease spread. However, as we will see in section \ref{sec:control}, we will be interested in minimizing cost functionals which (1) decrease with decreased disease spread and (2) increase with increasing controls, so there will be a tradeoff for the policy-maker to assess. 

Incorporating these controls leads to the system which will be of interest for the remainder of the manuscript: \begin{equation}\begin{aligned}\label{PDE_SYS}
    \heat S &= \xi b(x) -\beta (1-\alpha(x,t))SI_M -(\omu-\mu(x,t)) SN^* + \nu(x,t) S^*  - \delta S, \\
	\heat I &= \beta (1-\alpha(x,t))SI_M - \gamma I - (\omu-\mu(x,t)) IN^* + \nu(x,t) I^* - \delta I, \\
	\heat R &= \gamma I - (\omu-\mu(x,t)) RN^* + \nu(x,t) R^*- \delta R,\\
	\heat {S^*} &=(1-\xi)b(x) -\beta S^*I_M + (\omu-\mu(x,t)) SN^* - \nu(x,t) S^* - \delta S^*,\\
	\heat {I^*} &= \beta S^*I_M - \gamma I^* +(\omu-\mu(x,t))IN^* - \nu(x,t) I^* - \delta I^*, \\
	\heat {R^*}&= \gamma I^* + (\omu-\mu(x,t)) RN^* - \nu(x,t) R^* - \delta R^*.
\end{aligned}\end{equation} \CB{We note that the control map $\alpha(x,t)$ also now appears inside the actively mixing infectious population: $I_M = (1-\alpha(x,t))I+I^*$.} We assume that these dynamics take place in an open, bounded, simply connected domain $\Omega \subset \mathbb R^n$ which has Lipschitz boundary (in application $n = 2$ is perhaps most natural, but mathematically, we can just as easily analyze the problem for general $n \in \mathbb N$). We will be most interested in finite horizon control, so we assume $t \in [0,T]$. We assume all constant parameters are positive, and that the birth rate $b \in L^\infty(\Omega)$ is nonnegative. The system is equipped with nonnegative initial data $S_0,I_0,R_0,S^*_0,I^*_0,R^*_0 \in L^\infty(\Omega),$ and zero flux boundary conditions $\nabla S \cdot \B n = 0$ on $\Gamma := \partial \Omega$ (and similarly for all other populations).

To assist the reader with all of this notation, all variables, controls, and parameters are summarized in figure \ref{fig:params}, and the flow diagram for the system is included in figure \ref{fig:flowDiagram}. With this, we commence with analysis of \eqref{PDE_SYS}.

\begin{figure}[b!]
\centering
\begin{tabular}{|l l|}
\hline
    Quantity & Description\\
    \hline
    $S(x,t)$ & compliant susceptible population\\
    $I(x,t)$ & compliant infectious population\\
    $R(x,t)$ & compliant recovered population\\
    $S^*(x,t)$ & noncompliant susceptible population\\
    $I^*(x,t)$ & noncompliant infectious population\\
    $R^*(x,t)$ & noncompliant recovered population\\
    $N^*(x,t)$ & total noncompliant population ($N^* = S^* + I^* + R^*)$\\ 
    $b(x)$ & natural birth rate $(b \in L^\infty(\Omega), b \ge 0)$\\
    $\xi$ & portion of newly introduced population which is compliant ($\xi \in [0,1]$)\\
    $\delta$ & natural death rate ($\delta > 0$)\\
    $\beta$ & natural infectivity strength of the disease $(\beta > 0)$\\
    $\gamma$ & natural recovery rate for the disease $(\gamma > 0)$\\
    $\underline \alpha$ & baseline reduction in infectivity due to compliance ($\underline \alpha \in (0,1]$)\\
    $\alpha(x,t)$ & {\bf control variable}: further reduction in infectivity due to compliance $(\alpha(\cdot,\cdot) \in [\underline \alpha, 1])$\\
    $\omu$ & baseline infectivity strength of noncompliance ($\omu > 0$)\\ 
    $\mu(x,t)$ & {\bf control variable}: reduction in infectivity of noncompliance ($\mu(\cdot,\cdot) \in [0,\omu]$)\\
    $\overline \nu$ & maximal achievable rate of recovery from noncompliance ($\overline \nu > 0)$\\
    $\nu(x,t)$ & {\bf control variable}: recovery rate for noncompliance ($\nu(\cdot,\cdot) \in [0,\overline \nu]$)\\\hline
\end{tabular}
\caption{A list of state variables, control maps, and parameters for \eqref{PDE_SYS}.}
\label{fig:params}
\end{figure}

\begin{figure}[t!]
\centering
\includegraphics[width=0.5\textwidth]{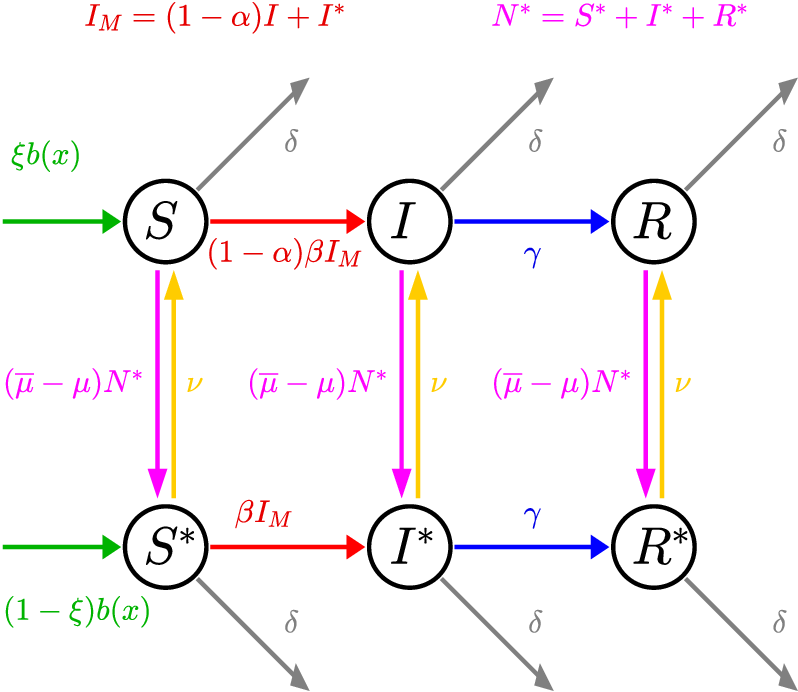}
\caption{The flow diagram for \eqref{PDE_SYS}. Any arrow flowing out of a population indicates flow proportional to the population it leaves. Here $I_M = (1-\alpha)I + I^*$ denotes the \emph{actively mixing infectious population} (i.e., those who contribute to disease spread).}
\label{fig:flowDiagram}
\end{figure}

\section{Global wellposedness with fixed controls} \label{sec:globalExistence}

Henceforth we define the solution vector $y = (y_1,y_2,y_3,y_4,y_5,y_6) = (S,I,R,S^*,I^*,R^*)$ and the control vector $u = (u_1,u_2,u_3) = (\alpha,\mu,\nu)$. 

In this paper we adopt M. Pierre's notion of classical solution defined in \cite[p. 419]{pierre}. For a constant control vector, global existence of a unique solution for \eqref{PDE_SYS} is proven in \cite{parkinson2023analysis}. The proof for nonnegative, bounded controls is essentially identical. We repeat the skeleton of the argument here, offer some formal reasoning, and refer the reader to \cite{parkinson2023analysis} for rigorous proofs. In what follows, for any $t>0$, we define the parabolic domain $Q_t := \Omega \times (0,t)$. Here $y$ satisfies the reaction-diffusion system \begin{equation}\label{VEC_SYS}
    \begin{split} 
    &y_t - D\Delta y = F(y,u) \,\,\, \mbox{in} \ Q_T \\
    &y(\cdot,0) = y_0 \in [L^\infty(\Omega)]^6\\
    &\nabla y_i \cdot \B n = 0 \,\,\, \text{on} \,\, \Gamma, \,\,\, \text{ for } \, i = 1,\ldots,6
    \end{split}
\end{equation} where $F = (f_1,f_2,f_3,f_4,f_5,f_6) : \mathbb R^6 \times \mathbb R^3 \to \mathbb R^6$ denotes the right hand side of \eqref{PDE_SYS} and $D$ is a diagonal matrix containing the diffusion coefficients (which for the purposes of this section, we will also relabel $d_1,d_2,d_3,d_4,d_5,d_6$). 

Assuming that initial data lies in $L^\infty(\Omega)$, local existence of a classical solution follows because the nonlinearities in $F$ are locally Lipschitz \cite[Theorem 14.2]{smoller}. 
Global existence then follows assuming we can establish \emph{a priori} $L^\infty$ bounds on finite time intervals $[0,T]$. In essence, the argument is that if the solution of \eqref{PDE_SYS} fails to exist past some finite time $T$, then one of components must blow up before (or at) time $T$ (see \cite[Lemma 1.1]{pierre} and the references therein). Thus the goal is to prove the required $L^\infty$ bound. The general strategy of the proof is to leverage mass bounds into $L^\infty$ bounds. We do this in a series of lemmas. The first lemma provides nonnegativity of solutions, assuming that initial data is nonnegative. 

\begin{lem}\label{nonnegativity}
System \eqref{VEC_SYS} $($or equivalently \eqref{PDE_SYS}$)$ preserves nonnegativity in the sense that if $y_0 \ge 0$, then $y(\cdot,t) \ge 0$ as long as the solution exists. $[$Note, that these are vector-valued quantities; we interpret these inequalities componentwise$]$.
\end{lem}

\begin{proof}
    This follows directly from the comparison principle, noting that $z\equiv (0,0,0,0,0,0)$ is a subsolution of \eqref{PDE_SYS} whenever $y_0 \ge 0$.
\end{proof} 

From here, mass bounds follow easily. We define the total population density $Y = \sum^6_{i=1} y_i$. Using the zero-flux boundary conditions, we have $$\int_\Omega \Delta y_i dx = 0$$ for each $i=1,\ldots, 6$. Thus adding all equations from \eqref{PDE_SYS} and integrating (in space) gives $$\frac d {dt} \|Y(t)\|_{L^1(\Omega)} = \|b\|_{L^1(\Omega)} - \delta \|Y(t)\|_{L^1(\Omega)}, $$ whereupon $$\|Y(t)\|_{L^1(\Omega)} = \frac{\|b\|_{L^1(\Omega)}}{\delta}(1 - e^{-\delta t}) + e^{-\delta t}\|Y_0\|_{L^1(\Omega)} \le \frac{\|b\|_{L^1(\Omega)}}{\delta} + \|Y_0\|_{L^1(\Omega)}.$$ Nonnegativity then guarantees $L^1$ bounds for the individual subpopulations. The majority of the work toward proving global existence is the spent upgrading these $L^1$ bounds into $L^\infty$ bounds. The next lemma establishes elementary $L^p$-norm bounds for reaction-diffusion equations. 

\begin{lem} \label{pnormbound}
    Let $f,v: L^\infty(Q_T)$ and suppose that $v$
    is a classical solution of \begin{align} &(\partial_t - d\Delta)v =f, \,\,\,\, {\rm{in}} \ Q_T, \\ &v(\cdot,t) = v_0 \in L^\infty(\Omega), \\
    &\frac{\partial v}{\partial n} = 0 \,\,\, {\rm{on}} \,\, \Gamma.\end{align} Then for any $p \in [1,\infty]$ and any $t \in [0,T]$,  \begin{equation} \label{eq:pnormbound} \|v(t)\|_{L^p(\Omega)} \le \|v_0\|_{L^p(\Omega)} + \int^t_0  \|f(s)\|_{L^p(\Omega)} ds.\end{equation}
\end{lem}

\begin{proof}
In the $p = \infty$ case, the proof follows by defining the function $u(x,t) = \|v_0\|_{L^\infty(\Omega)} + \int^t_0 \|f(s)\|_{L^\infty(\Omega)} ds - v(x,t)$ and proving that the diffusion operator preserves nonnegativity (which can be proven in the spirit of lemma \ref{nonnegativity}). 

In the $p \in [1,\infty)$ case, one multiplies the equation by $p\abs{v}^{p-2}v$ and integrates to derive the desired inequality. 

Full proofs in each case are contained in \cite{parkinson2023analysis}.
 \end{proof}

 \noindent{\bf Remark.} Note that by the comparison principle, lemma \ref{pnormbound} still holds if the assumptions are relaxed to $(\partial_t - d\Delta)v \le f$ and $v(\cdot,0) \le v_0$. \\

 \begin{cor} \label{corToLem}
With the same assumptions as in lemma \ref{pnormbound}, we have $$\|v\|^p_{L^p(Q_t)} \le C(1+\|f\|^p_{L^p(Q_t)})$$ where $C$ is a positive constant depending on $p$ and $T$.
 \end{cor}

 \begin{proof} When $p > 1$, using Jensen's inequality, we see $$\left(\int^t_0 \|f(s)\|_{L^p(\Omega)}ds\right)^p \le t^{p-1} \int_0^t \|f(s)\|^p_{L^p(\Omega)}ds \leqslant C\|f\|^p_{L^p(Q_t)},$$ where $C$ is a constant depending on $T$ and $p$ (when $p = 1$, this trivially holds with equality). Thus, raising \eqref{eq:pnormbound} to the power $p$ and recalling the standard inequality $(a+b)^p \le 2^p \max\{a^p,b^p\} \le 2^p(a^p+b^p)$ which holds for $a,b \ge 0, p \ge 1,$ we have  $$\|v(t)\|^p_{L^p(\Omega)} \le 2^p(\|v_0\|^p_{L^p(\Omega)} + C\|f\|^p_{L^p(Q_t)}) \leqslant C(1+\|f\|^p_{L^p(Q_t)}).$$ Integrating on $[0,t]$ yields the desired inequality.  
 \end{proof}

 Finally, one of the main difficulties that arises is that the diffusion coefficients in \eqref{PDE_SYS} could be different. If all coefficients were the same, we could add equations together to strategically eliminate nonlinear terms and establish $L^\infty$ bounds one-by-one.  Because this may not be possible, we need a lemma that maintains $L^p$ control of a solution to a reaction-diffusion equation if the diffusion coefficient is changed. This is a key lemma in the ensuing theorem which establishes $L^p$ mass bounds for all $p \in [1,\infty)$.
 
 \begin{lem} \label{keyLemma}
 Let $v,w \in L^p(Q_T)$ $(p \in [1,\infty))$
 satisfy $v(\cdot,0), w(\cdot,0) \in L^\infty(\Omega)$ and \begin{equation}\label{eq:vwbound} (\partial_t - d\Delta)v \le c_1 \partial_t w + c_2 \Delta w, \,\,\,\,\,\, \text{in} \,\,\,\, Q_T, \qquad \text{for some } c_1, c_2 > 0.\end{equation} Then for any $t \in [0,T]$, we have $$\|v\|_{L^p(Q_t)} \le C(1+\|w\|_{L^p(Q_t)})$$ where $C$ is a constant depending on the ambient parameters and initial data.
 \end{lem}

 \begin{proof}
      The strategy is to use the dual definition of the $L^p$-norm: $$\|v\|_{L^p(Q_t)} = \sup \left\{\langle v,g\rangle \, : \, g \in L^q(Q_t), \, \|g\|_{L^q(Q_t)} \le 1, \, g \ge 0\right\}.$$ Now for any such $g$, formulate a dual problem which runs backwards in time: \begin{align} -&\varphi_t - d \Delta \varphi = g \,\,\,\,\, \text{ on } \Omega \times [0,t), \\
     &\varphi(\cdot,t) = 0. 
     \end{align} Multiply \eqref{eq:vwbound} by the nonnegative solution $\varphi$ of this dual problem, integrate to pass all derivatives to $\varphi$ and use the following parabolic regularity estimate found in \cite[Chap. 4, \S9]{lady}, \cite[Chap. 9, \S2]{WuYinWang} $$\pnorm{\partial_t \varphi}{q}{Q_t} + \pnorm{\Delta \varphi}{q}{Q_t} + \sup_{s \in [0,t)} \pnorm{\varphi(s)}q\Omega \le C \pnorm g q{Q_t}$$ to derive the desired inequality.  
 \end{proof}

 We make one last observation before we prove an $L^\infty$ bound which will provide global existence. By assumptions ($i$)-($iii$) on the controls, there are constants $c_{\text{min}}, c_{\text{max}}$ such that $$ 0\le  c_{\text{min}} \le \alpha(x,t), \,\, \,u(x,t), \,\,\, \nu(x,t) \le c_{\text{max}}.$$ This assumption will also come into play in the optimal control analysis in the ensuing section. \CB{We remind the reader that $y = (y_1,y_2,y_3,y_4,y_5,y_6) = (S,I,R,S^*,I^*,R^*)$.} Given this, we note that as long as $y \ge 0$, $F(y) = (f_1(y),f_2(y),f_3(y),f_4(y),f_5(y),f_6(y))$ satisfies \begin{equation} \label{eq:MBS} \begin{split} f_1(y) &\le \xi b(x) + c_{\text{max}}y_4, \\
 f_1(y) + f_2(y) &\le \xi b(x) + c_{\text{max}}(y_4 + y_5), \\
  f_1(y) + f_2(y) +f_3(y) &\le \xi b(x) + c_{\text{max}}(y_4 + y_5 + y_6), \\
   f_1(y) + f_2(y) + f_3(y) + f_4(y) &\le \hphantom{\xi}b(x) + c_{\text{max}}(y_5 + y_6), \\
   f_1(y) + f_2(y) + f_3(y) + f_4(y)+f_5(y) &\le \hphantom{\xi}b(x) + c_{\text{max}}y_6, \\
   f_1(y) + f_2(y) + f_3(y) + f_4(y)+ f_5(y)+f_6(y) &\le \hphantom{\xi}b(x).
 \end{split}\end{equation} The important point is that when successively adding the right hand sides of \eqref{PDE_SYS}, we can bound the partial sums by a constant plus a \emph{linear} function of $y$, while the nonlinear terms either cancel in the summation, or can be dropped because they are nonpositive. It is this structure---along with lemma \ref{keyLemma}---that enables us to prove the next theorem. 

 \begin{theorem} \label{globalLpBound}
     Fix any $p \in [1,\infty)$. The unique local-in-time classical solution of \eqref{VEC_SYS} remains bounded in $L^p(Q_t)$ on any finite subinterval of the maximum interval of existence. That is, for any $T>0$, if the classical solution $y$ of \eqref{VEC_SYS} exists on $[0,T)$, then there is $M >0$ depending on $T$ and the ambient parameters such that $$\max_{1\le i \le 6} \|y_i\|_{L^p(Q_t)} \le M, \,\,\,\,\, \text{ for all } \,\,\, t \in [0,T).$$ 
 \end{theorem}

 \begin{proof} The full proof is provided in \cite{parkinson2023analysis}. For completeness, we repeat the outline here.
 
    The proof proceeds from the observation that the system is ``mass-bounded" in the sense of \eqref{eq:MBS}. Given this, we define auxiliary functions $z_i,\,\, i=1,\ldots, 6$ such that \begin{equation} \label{eq:zFuncs} \begin{split} (\partial_t -d_1 \Delta)z_1 &= \xi b(x) + c_{\text{max}}y_4, \\
    (\partial_t -d_2 \Delta)z_2 &= \xi b(x) + c_{\text{max}}(y_4 + y_5), \\
  (\partial_t -d_3 \Delta)z_3 &= \xi b(x) + c_{\text{max}}(y_4 + y_5 + y_6), \\
   (\partial_t -d_4 \Delta)z_4 &= \hphantom{\xi}b(x) + c_{\text{max}}(y_5 + y_6), \\
   (\partial_t -d_5 \Delta)z_5 &= \hphantom{\xi}b(x) + c_{\text{max}}y_6, \\
   (\partial_t -d_6 \Delta)z_6 &= \hphantom{\xi}b(x).
 \end{split} \end{equation} with zero-flux boundary data and homogeneous initial data $z_i(x,0)\equiv 0$ for all $i=1,\ldots,6$. For the remainder of this proof, we fix an arbitrary $t>0$ and an arbitrary $p \in [1,\infty)$, and $C$ will denote a positive constant which changes from line to line and depends on the ambient parameters including $T$. 
 
 With these definitions and using \eqref{eq:MBS}, we see $$(\partial_t - d_1 \Delta)[y_1-z_1] \le 0$$ thus from lemma \ref{pnormbound}, we have $$\|y_1(t) - z_1(t) \|_{L^p(\Omega)} \le \|y_1(0)\|_{L^p(\Omega)},$$ so $$\|y_1(t)\|_{L^p(\Omega)} \le C(1+\|z_1(t)\|_{L^p(\Omega)}).$$ As in the proof of corollary \ref{corToLem}, this yields $$\|y_1\|^p_{L^p(Q_t)} \le C(1+\|z_1\|^p_{L^p(Q_t)}).$$ \CB{Next, note that $$(\partial_t - d_2 \Delta)[y_2 - z_2] + (\partial_t -d_1 \Delta )y_1 = f_1(y) + f_2(y) - (\partial_t - d_2 \Delta)z_2 \le 0 $$ where the inequality follows from \eqref{eq:MBS} and the definition of $z_2$ in \eqref{eq:zFuncs}.} Thus $$(\partial_t - d_2 \Delta)[y_2 - z_2] \le -(\partial_t - d_1 \Delta)y_1.$$ Applying lemma \ref{keyLemma}, we have $$\|y_2-z_2\|_{L^p(Q_t)} \le C(1+\|y_1\|_{L^p(Q_t)})$$ whereupon the reverse triangle inequality and the bound on $\|y_1\|^p_{L^p(Q_t)}$ yield $$\|y_2\|^p_{L^p(Q_t)} \le C(1+\|z_1\|^p_{L^p(Q_t)} + \|z_2\|^p_{L^p(Q_t)}).$$

Continuing in this same manner, we arrive at bounds $$\|y_i\|^p_{L^p(Q_t)} \le C\left(1+\sum^i_{j=1} \|z_j\|^p_{L^p(Q_t)}\right)$$ for each $i=1,\ldots,6.$

Thus, defining $$
    P(t) = \sum^{n}_{j=1} \|y_j\|^p_{L^p(Q_t)}, \hspace{1cm}
    Z(t) = \sum^{n}_{j=1} \|z_j\|^p_{L^p(Q_t)},
$$ we have the inequality \begin{equation} \label{PZbd} P(t) \le C(1+Z(t)).\end{equation} However, each function $z_i$ satisfies $$(\partial_t - d_i\Delta)z_i \le C\left(1+\sum^n_{j=1} y_j\right),$$ so applying lemma \ref{pnormbound} and \ref{corToLem} then integrating on $[0,t]$ yields $$\|z_i\|^p_{L^p(Q_t)} \le C\left(1+\int^t_0 P(s)ds.\right)$$ Inserting this into \eqref{PZbd}, we have $$P(t) \le C\left(1+\int^t_0P(s)ds\right), \,\,\,\,\,\, t \in [0,T).$$ Gronwall's inequality then gives boundedness of $P(t)$, and thus of $\|y_i\|_{L(Q_t)}$ for each $i=1,\ldots,6$.
 \end{proof}

 Finally, we use classical results regarding parabolic regularity to conclude. 

\begin{theorem} \label{thm:globalExistence}
    The unique local-in-time classical solution of \eqref{VEC_SYS} remains bounded in $L^\infty(Q_t)$ on any finite subinterval of the maximum interval of existence, and thus exists globally in time. 
\end{theorem}

\begin{proof}
A classic result regarding parabolic regularity (see \cite[Ch. 4, \S 9]{lady}, \cite[Chap. 9, \S2]{WuYinWang} for example) states that for each $i=1,\ldots,6$ and any $p\in [1,\infty)$, \begin{equation} \label{eq:paraEst}\|\partial_t y_i\|_{L^p(Q_t)} + \|\nabla y_i\|_{L^p(Q_t)} \le C(\|y_i(0)\|_{L^p(Q_t)} + \|f_i(y)\|_{L^p(Q_t)})\end{equation} where $f_i$ is the right hand side of the corresponding equation. Since all nonlinearities are quadratic, we have $$\abs{f_i(y)} \le C\left(1 + \sum^6_{j=1} \abs{y_i} + \sum^6_{j=1} \abs{y_i}^2\right),$$ so $$\|f_i(y)\|^p_{L^p(Q_t)} \le C\left(1 + \sum^6_{j=1} \|y_i\|^p_{L^p(Q_t)} + \sum^6_{j=1} \|y_i\|^{2p}_{L^{2p}(Q_t)}\right).$$ Inserting this bound in \eqref{eq:paraEst} and applying theorem \ref{globalLpBound} shows that $\partial_t y_i, \nabla y_i$ are bounded in $L^p(Q_t)$, and thus $y_i \in W^{1,p}(Q_t)$ for any $p \in [1,\infty).$ Choosing $p > n$, the Sobolev embedding theorem guarantees that $y_i \in L^\infty(Q_t)$ whereupon we have global in-time existence. 
\end{proof}

\section{Optimal Control} \label{sec:control}

Recall that we work with the state variables $y = (S,I,R,S^*,I^*,R^*)^\top$ and control maps $u = (\alpha, \mu, \nu)$. It follows from \cite[Eq. (1.5), p. 419]{pierre} that
%
%
%
if $y$ is a classical solution then $y \in [W_{loc}^{1,p}(0,T; W^{2,p}(\Omega))]^6$ which in particular means that both $y_i$ and $D_x y_i$ have traces in $L_{loc}^p((0,T) \times \Gamma)$, $i = 1,...,6.$ Moreover, from \cite[Theorem 3.5, p. 11]{parkinson2023analysis} it follows that given $y(0) \in [L^\infty(\Omega)]^6$ with each component being non-negative and $u \in [L^\infty(\mathbb{R}_+ \times \Omega)]^3$, there exists a unique $y \in [W_{loc}^{1,p}(\mathbb{R}_+;W^{2,p}(\Omega))]^6$.

\subsection{Existence of optimal control}\label{eoc}

With $y = (y_1,y_2,y_3,y_4,y_5,y_6) = (S,I,R,S^*, I^*, R^*)$ and $u = (u_1, u_2, u_3) = (\alpha, \mu, \nu)$ we introduce state affine polynomials $p_i: \mathbb{R}^6 \to \mathbb{R}$ of the form $p_i(\vec{x}) = \vec{a}_i^\top \vec{x} + k_i$ with $\vec{a}_i \in \mathbb{R}^6$ and $k_i \in \mathbb{R}$ and control affine polynomials $q_i: \mathbb{R}^3 \to \mathbb{R}$ of the form $q_i(\vec{x}) = \vec{c}_i^\top \vec{x} + l_i$ with $\vec{c} \in \mathbb{R}^3$ and $l_i \in \mathbb{R}.$  

Choosing a time horizon $T>0$, a number $m_c \in \mathbb{N}$ and a vector $P_s = (p_s^1, \cdots, p_s^m)$ ($p_s^i \geqslant 1)$ for some $m_s \in \mathbb{N}$ fixed we define our cost functional $\mathcal{J}(y,u)$ to be of the form 
\begin{align}\label{e.w10151_gen}
    \mathcal{J}(y,u) 
    &= \sum\limits_{i = 1}^{m_s} \dfrac{\lambda_i}{p_s^i}\|p_i(y)\|_{L^{p_s^i}(Q_T)}^{p_s^i} + \sum\limits_{i = 1}^{m_c} \dfrac{\zeta_i}{2}\|q_i(u)\|_{L^{2}(Q_T)}^{2} = \sum\limits_{i=1}^{m_s} \mathcal{J}_s^i(y,u) + \sum\limits_{i=1}^{m_c} \mathcal{J}_c^i(u).
\end{align}
For an example of a meaningful $\mathcal{J}$ we can take $m_s = 2$,  $p_s = (1,1), m_c = 3$ and the polynomials $p_1(\vec{x}) = x_2 + x_5, p_2(\vec{x}) = x_4 + x_5 + x_6$, $q_1(\vec{x}) = x_1 - \underline \alpha$, $q_2(\vec{x}) = x_2$ and $q_3(\vec{x}) = x_3.$ See next section for a biological interpretation of such cost functional and simulations using it. 
The lemma below, which we do not prove, follows from properties of norms.
\begin{lem}
    The functional $\mathcal{J}: [L^{\max\{p_s^i\}}(Q_T)]^6 \times [L^2(Q_T)]^3 \to \mathbb{R}^+$ is weakly lower semi-continuous.
\end{lem}

The \CB{optimal} control problem is then formulated as
\begin{equation}\label{e.w09181_cpyu}
    \min \{\mathcal{J}(y,u); u \in U^{\CB{p}}_{ad} \ \mbox{and} \ y \ \mbox{solves \eqref{PDE_SYS}}\}
\end{equation}
where \begin{equation}
    \label{Uad} U^{\CB{p}}_{ad} \equiv \{u \in [L\CB{^p}(Q_T)]^3; \ 0 \leqslant A_i \leqslant u_i(x,t) \leqslant B_i, \,\,\,\, i = 1,2,3\} \subset [L^\infty(Q_T)]^3
\end{equation} \CB{with $A_i, B_i \in \mathbb{R}$ and $p \in [1,\infty)$.}   Note that $U^{\CB{p}}_{ad}$ is a closed, bounded, convex subset of $[L^p(Q_T)]^3 \equiv U^{\CB{p}}$ (though not a subspace). To recall, for our purposes, we have $u_1 = \alpha, u_2 = \mu, u_3 = \nu.$ Thus \eqref{Uad} specifies that each of the control maps takes values within some bounded, nonnegative interval, as designed in section \ref{modeling} above.

\CB{\begin{remark}
    The inclusion to $[L^\infty(Q_T)]^3$ in \eqref{Uad} holds for any $p \geqslant 1.$ This is in fact necessary because we consider $L^2$ norms of the controls in the objective function \eqref{e.w10151_gen} and because we need $p>n$ to guarantee that solutions are in $L^\infty$ (see Theorem \ref{thm:globalExistence}). However, we still need to restrict the definition of admissible set to finite $p$ for topological reasons. 
    In order to avoid confusion, we use the superscript $p$ in $U_{ad}^p$ to indicate what topology is being used in the control set.
\end{remark}}

The wellposedness theory defines a control-to-state operator $\cM: U_{ad}^{\CB{p}} \to W_{loc}^{1,p}(0,T;\tilde Y)$ with $$\tilde Y = \{y \in \CB{[W^{2,p}(\Omega)]^6}; \nabla y \cdot \mathbf{n} = 0\}.$$ In particular, due to uniqueness, problem \eqref{e.w09181_cpyu} can thus be reduced to \begin{equation}\label{ct1}    \min_{u \in U_{ad}^{\CB{p}}} J(u) \equiv \mathcal{J}(\cM(u),u).
\end{equation}

Our next goal is to show that an optimal control exists. The standard \emph{direct} method usually takes the following path: \begin{itemize}
    \item[\bf(i)] One assumes $U_{ad}^{\CB{p}} \neq \emptyset$ so $\cM^{-1}(U_{ad}^{\CB{p}}) \neq \emptyset$, then the $\inf\{J(u); u \in U_{ad}^{\CB{p}}\}$ exists. Call it $d \in \mathbb{R}$.
    \item[\bf(ii)] Properties of the infimum provide a sequence $u_n \in U_{ad}^{\CB{p}}$ such that $J(u_n) \to d.$
    \item[\bf(iii)] One needs the structure of $U_{ad}^{\CB{p}}$ to allow the sequence $(u_n)$ (or a subsequence) to have a limit in $U_{ad}^{\CB{p}}$. This is, of course, closely connected to compactness properties of $U_{ad}^{\CB{p}}$ which is usually too much to ask for when infinite dimensional spaces are involved. The reasonable assumption is then weak (sequential) compactness of $U_{ad}^{\CB{p}}$ which is usually achieved via closedness and reflexiveness. In our case we even have boundedness of $U_{ad}^{\CB{p}}$, so our $U_{ad}^{\CB{p}}$ is weak sequentially compact, and any such sequence can be assumed to be weakly convergent (perhaps along a subsequence) to some ${u^\circ} \in U_{ad}^{\CB{p}}.$
    \item[\bf(iv)] We know that $J(u_n) \to d$ and $u_n \rightharpoonup {u^\circ} \in U_{ad}^{\CB{p}}.$ The next goal is then to relate $J({u^\circ})$ with $d.$ If one is able to show that $J({u^\circ}) = d$, then ${u^\circ}$ is an optimal control (and in this case we can replace $\inf$ by $\min$). The usual (minimal) assumption one makes on $J$ is that $u \mapsto J(u)$ is weakly lower semi-continuous (w.l.s.c.), so that if $u_n \rightharpoonup {u^\circ}$ in $U_{ad}^{\CB{p}}$ then $$J({u^\circ}) \leqslant \liminf\limits_{n\to\infty}J(u_n).$$ It is clear that this would imply that ${u^\circ}$ is an optimal control. In our case, however, $J(u) = \mathcal{J}(\cM(u),u)$, hence any continuity property of $J$ depends on the regularity properties of $\cM.$ Thus we need to guarantee that $\cM$ is weakly continuous, i.e., if $u_n \rightharpoonup {u^\circ}$ in $U$ then $\cM(u_n) \rightharpoonup \cM({u^\circ})$ in the appropriate space required by the cost functional.
\end{itemize}

The main Theorem in this subsection is as follows (see remark \ref{wcs}).
\begin{theorem}\label{wtos}
Let $Z = L^{\max\{p_s^i\}}(Q_T)$. The map $\cM:U_{ad}^{\CB{p}} \to \CB{Z^6}$ is {weak--to--strong} continuous, i.e., given a sequence $u_n$ in $U_{ad}^{\CB{p}}$ 
	\begin{equation}
	u_n \rightharpoonup u \ \mbox{\rm in} \ U^{\CB{p}} \Longrightarrow \cM(u_n)\to \cM(u) \ \mbox{\rm in} \ Z^{\CB{6}}.
	\end{equation}
\end{theorem}

The above theorem is bit more than what we need to conclude existence of optimal control. The proposition below is therefore a corollary of Theorem \ref{wtos} along with steps (i)--(iv) discussed above.

\begin{proposition}[\bf Existence of optimal control] Problem \eqref{ct1} admits a solution.
\end{proposition}

The remainder of this subsection is dedicated to the proof of Theorem \eqref{wtos}. We use the notation of \eqref{VEC_SYS}. In the next proposition, we put \begin{equation}
    \label{def_Y} Y \equiv W^{1,p}(0,T;L^p(\Omega)) \cap L^p(0,T;W^{2,p}(\Omega)) \cap C([0,T];W^{1,p}(\Omega)).
\end{equation}

\begin{proposition}
    For $p \geqslant 2$, $\cM: U_{ad}^{\CB{p}} \to \CB{Y^6}$ is Lipschitz continuous.
\end{proposition}
\begin{proof}
    Let $u = (\alpha_u, \mu_u, \nu_u)^\top, v=(\alpha_v, \mu_v, \nu_v)^\top \in U_{ad}^{\CB{p}}$ and $z = \cM(u) - \cM(v).$ Then $z$ solves 
    \begin{equation}\label{M_char0}
        z_t - D\Delta z = F(\cM(u),u) - F(\cM(v),v)
    \end{equation} with zero initial condition and $\nabla z \cdot \mathbf{n} = 0.$ 
    
By direct computation, one sees that there are matrices $M_c = M_c(\cM(u),\cM(v),u,v)$ and $M_e = M_e(\cM(u),\cM(v),u,v)$ of order $6 \times 3$ and $6 \times 6$ respectively such that \begin{equation}\label{M_char}
            F(\cM(u),u) - F(\cM(v),v) = M_c(u-v) + M_ez.
        \end{equation} Moreover, all entries of both $M_c$ and $M_e$ belong to $L^\infty(Q_T)$ and by defining $$\|M\|_{\infty} = \max_i \sum_j \|m_{ij}\|_{L^\infty(Q_T)}$$ there exist constants $C_c, C_e > 0$ such that $\|M_c\|_{\infty} \leqslant C_c$ and $\|M_e\|_{\infty} \leqslant C_e.$ The proof of identity \eqref{M_char} is tedious but straightforward if one computes it by hand and can be considerably simplified if one uses the mean value theorem. We do not include it here.

    It follows by \eqref{M_char0} and \eqref{M_char} that $z$ solves 
    \begin{equation}\label{z_eq}
        z_t - D\Delta z = M_c(u-v) + M_ez
    \end{equation} with zero initial condition and $\nabla z \cdot \mathbf{n} = 0.$ The classic parabolic estimate \cite{lady}[Theorem 9.1, p. 341] gives \begin{equation}
        \|z\|_{\CB{[L^p(0,T;W^{2,p}(\Omega))]^6}}^p + \|z_t\|_{\CB{[L^p(Q_T)]^6}}^p \leqslant C\|u-v\|_{U^{\CB{p}}}^p.
    \end{equation} The result then follows from the continuous embedding $$H^1(0,T;\CB{L^p(\Omega)}) \cap L^2(0,T;\CB{W^{2,p}(\Omega))} \hookrightarrow C([0,T];\CB{W^{1,p}(\Omega))}.$$ \end{proof}
    
A corollary of the previous proposition is the \emph{strong--to--strong} continuity of the map $\cM: U_{ad}^{\CB{p}} \to Y^{\CB{6}}$. 

\begin{lem}
	For $p \geqslant \max\{2,\max\{p_s^i\}\}$, the map $\cM: U_{ad}^{\CB{p}} \to Z^{\CB{6}} \ (Z = L^{\max\{p_s^i\}}(Q_T))$ is weakly closed. That is, given a sequence $(u_n) \in U_{ad}^{\CB{p}}$  one has
\begin{equation}\label{weaklyclosed}
	u_n \rightharpoonup u \ \mbox{\rm in} \ U^{\CB{p}} \ \mbox{\rm and} \ \cM(u_n) \rightharpoonup v \ \mbox{\rm in} \ Z^{\CB{6}} \Longrightarrow u \in U_{ad}^{\CB{p}} \ \mbox{\rm and} \ \cM(u) = v.
\end{equation}
\end{lem}
\begin{proof}
    Let $(u_n)$ be a sequence in $U_{ad}^{\CB{p}}$ such that $u_n \rightharpoonup u $ in $U^{\CB{p}}$ and $\cM(u_n) \rightharpoonup v$ in $Y^{\CB{6}}.$ We have $u \in U_{ad}^{\CB{p}}$ and, from Lipschitz continuity of $\cM: U_{ad}^{\CB{p}} \to Y^{\CB{p}}$, $(\cM({u_n}))$ is uniformly bounded in $Y^{\CB{6}}$. The embeddings 
	\begin{equation}\label{embrec}
		Y \hookrightarrow \mathcal{Y} \equiv W^{1,p}(0,T;L^2(\Omega)) \cap L^p(0,T;W^{2,p}(\Omega)) \cap L^p(0,T;W^{1,p}(\Omega)) \hookrightarrow Z
	\end{equation}
	imply that $(\cM(u_n))$ is also uniformly bounded in $\mathcal{Y}^{\CB{6}}$. By reflexiveness we can assume that (along a subsequence if necessary)  there exists $w \in \mathcal{Y}^{\CB{6}}$ such that $\cM(u_n) \rightharpoonup w$ in $\mathcal{Y}^{\CB{6}}$. By the second embedding in \eqref{embrec} and by the uniqueness of the weak limit, we have $w = v$ in $Z^{\CB{6}}$. One can also show that $v$ is a weak solution of \eqref{VEC_SYS} in the semigroup sense, which is unique. Then $v = \cM(u).$

\end{proof}
\begin{proposition}
	For $p \geqslant \max\{2,\max\{p_s^i\}\}$, the map $\cM: U_{ad}^{\CB{p}} \to \mathcal{Y}^{\CB{6}}$ is weak--to--weak continuous.
\end{proposition}
\begin{proof}
	Let $(u_n)$ be a sequence in $U_{ad}^{\CB{p}}$ such that $u_n \rightharpoonup u$ in $U^{\CB{p}}$. Then $u \in U_{ad}^{\CB{p}}$ and we claim that the sequence $(\cM(u_n))$ converges \emph{weakly} to $\cM(u)$ in $Z^{\CB{6}}$.
	
	Indeed, $(\cM(u_n))$ is uniformly bounded in $\mathcal{Y}^{\CB{6}}$. Let $(\cM(u_{n_k}))$ be a subsequence of $(\cM(u_n))$. Then, there exists a further subsequence $(\cM(u_{n_{k_j}}))$ which converges weakly to some $v$ in $\mathcal{Y}^{\CB{6}}$ and $v = \cM(u)$ from weak closedness. Therefore, Uryson's subsequence principle yields $\cM(u_n) \rightharpoonup \cM(u)$ in $Z^{\CB{6}}.$
\end{proof}
\begin{cor}\label{wtoscor}
	For $p \geqslant \max\{2,\max\{p_s^i\}\}$, the map $\cM: U_{ad}^{\CB{p}} \to Z^{\CB{6}}$ is weak--to--strong continuous.
\end{cor}
\begin{proof}
	This follows from the compactness of the embedding $\mathcal{Y} \hookrightarrow Z$.
\end{proof}

\begin{remark}[\bf the worst case scenario]\label{wcs} An attentive reader may ask two natural questions. First, why are we proving weak-to-strong continuity when it seems that weak continuity is enough $($given item \emph{(iv)} of our discussion preceding Theorem \emph{\ref{wtos}}$)$? Secondly, what happens in the case $p_s^i = 1$ for some (or all) $i$? 

In fact the first question leads to the second and the second leads to Corollary {\ref{wtoscor}}. It is correct that weak continuity is enough for concluding Theorem \emph{\ref{wtos}}. However, in most cases the lack of higher regularity of solutions along with the fact that $L^1$ spaces are not reflexive causes the map $\cM$ to fail to exhibit weak continuity. That is why, in our case, we need to show that $\cM$ is weak-to-strong continuous with values in a space where solutions have higher integrability. This in turn yields, in one shot, weak closedness and weak-to-strong continuity of $\cM$ as a map from $U_{ad}^{\CB{p}}$ to $[L^1(\Omega)]^6$. 

\end{remark}

\subsection{Optimality conditions}

In this section, we derive first order optimality conditions which will in turn inform the numerics used for simulation in the next section. 

\subsubsection{Derivation of the Lagrangian}

 First, we formulate the Lagrangian function associated to our problem. Given the cost function that we are interested in, we perform this derivation (as well as the characterization of the control) under the following assumption.

\begin{ass}\label{ass_pol}
    The polynomials $p_i$, $q_i$ are such that $p_i(y),q_i(u) \geqslant 0$ for all $y,u,i.$
\end{ass} 

Since we have three constraints (the PDE, the boundary condition, and the initial condition) our Lagrange multiplier is of the form 
$\Phi = (\Phi_s,\Phi_\partial,\Phi_0) = ((\Phi_s^1, \cdots, \Phi_s^6), (\Phi_\partial^1, \cdots, \Phi_\partial^6), (\Phi_0^1,\cdots, \Phi_0^6))$ where, for each $i$, $\Phi_s^i: Q_T \to \mathbb{R}$, $\Phi_\partial^i: \Sigma_T := (0,T) \times \Gamma \to \mathbb{R}$ and $\Phi_0^i: \Omega \to \mathbb{R}.$ We formally define the Lagrangian function as 
\begin{align}
    \mathcal{L} = \mathcal{L}(y,u,\Phi) &= \mathcal{J}(y,u) - \int_{Q_T} \Phi_s \cdot (y_t - D\Delta y - F(y,u))dQ_T \nonumber \\ & - \int_{\Sigma_T} \Phi_\partial \cdot (\nabla y \cdot \mathbf{n})d\Sigma_T - \int_\Omega \Phi_0 \cdot (y(0)-y_0) d\Omega.
\end{align} 

Assuming for a moment that we have no regularity restrictions on $\Phi$ we can integrate the expression above by parts (twice) to get 
\begin{align}
    \mathcal{L}(y,u,\Phi) &= \mathcal{J}(y,u) + \int_{Q_T} ({\Phi_s})_t\cdot y dQ_T -\int_\Omega \{\Phi_s(T)y(T)-\Phi_s(0)y(0)\}d\Omega \nonumber \\&+ \int_{Q_T}\Phi_s \cdot F(y,u)dQ_T + \int_{Q_T} D\Delta \Phi_s y dQ_T \nonumber \\ &+ \int_{\Sigma_T} D[(\nabla y \cdot \mathbf{n})\Phi_s - (\nabla \Phi_s \cdot \mathbf{n})y]d\Sigma_T\nonumber  - \int_{\Sigma_T} \Phi_\partial \cdot (\nabla y \cdot \mathbf{n})d\Sigma_T - \int_\Omega \Phi_0 \cdot (y(0)-y_0) d\Omega.
\end{align} 
From the \CB{formal Lagrange principle (see \cite[Chap. 2]{Troltzsch})}, we expect an optimal control pair $({y^\circ}, {u^\circ})$ to satisfy the optimality conditions of the problem.
\begin{equation}
    \label{lag_pro} \min\limits_{u \in U_{ad}^{\CB{p}}} \mathcal{L}(y,u,\Phi)
\end{equation} with $y$ unconstrained. 
Under Assumption \ref{ass_pol}, and using a test function $h$ (we will specify the test space later), we have 
 \begin{align} 
    D_y\mathcal{L}({y^\circ},{u^\circ},\Phi)h &= \sum\limits_{i = 1}^{m_s} \lambda_i \int_{Q_T}p_i({y^\circ})^{p_s^i-1}\vec{a_i}^\top \cdot h dQ_T \\ &- \int_{Q_T} ({\Phi_s})_t\cdot h dQ_T + \int_\Omega \{\Phi_s(T)h(T)-\Phi_s(0)h(0)\}d\Omega \\ &+ \int_{Q_T}\Phi_s \cdot D_yF({y^\circ},{u^\circ})hdQ_T + \int_{Q_T} D\Delta \Phi_s \cdot h dQ_T \\ &+ \int_{\Sigma_T} D[(\nabla h \cdot \mathbf{n})\Phi_s - (\nabla \Phi_s \cdot \mathbf{n})h]d\Sigma_T  - \int_{\Sigma_T} \Phi_\partial \cdot (\nabla h \cdot \mathbf{n})d\Sigma_T - \int_\Omega \Phi_0 \cdot h(0) d\Omega. \label{lagrr}
\end{align} 

\begin{remark}
    Notice that $D_yF$ appeared in the expression above. It is important to mention here that the nonlinearity $F(y,u)$ can be seen as a Nemytskii operator associated to a map $$F: \mathbb{R}_+ \times \mathbb{R}^n \times \mathbb{R}^6 \times \mathbb{R}^3 \to \mathbb{R}^6.$$ Here, the first two components are the independent variables $(t,x)$, and the last two are the dependent variables $(y,u)$ seen as vectors in their respective Euclidean spaces only and not as vector-functions. It is not surprising then that differentiability properties of the map $u \mapsto \cM(u)$ are closely related to differentiability properties of $F$.

It is not difficult to see that as a map from $Q_T \times \mathbb{R}^6 \times \mathbb{R}^3 \to \mathbb{R}^6$, $F$ is smooth in both $y$ and $u$ (in fact it is a polynomial map). Hence, since $u \to \cM(u)$ is well defined and solutions are $L^\infty$ we can also take derivatives in the $L^\infty$ topology, under some basic properties which is satisfied by our $F$ (see \cite[Chap. 4]{Troltzsch}). Hence the matrix $D_yF(y^\circ, u^\circ)$ above is well defined as has all its entries in $L^\infty(Q_T).$
\end{remark}

 We now want to use the equation $D_y\mathcal{L}({y^\circ},{u^\circ},\Phi)h = 0$ to characterize the multiplier $\Phi.$ We start by taking $h \in C_0^\infty((0,T);[C_0^\infty(\Omega)]^6)$ in \eqref{lagrr}. This means that all the boundary terms (both Dirichlet and Neumann) and the initial/terminal terms disappear. Hence,
\begin{align}
    0 &= \sum\limits_{i = 1}^{m_s} \lambda_i \int_{Q_T}p_i({y^\circ})^{p_s^i-1}\vec{a_i}^\top \cdot h dQ_T + \int_{Q_T} (\Phi_s)_t\cdot h dQ_T \nonumber \\&+ \int_{Q_T}\Phi_s \cdot D_yF({y^\circ},{u^\circ})hdQ_T + \int_{Q_T} D\Delta \Phi_s \cdot h dQ_T.
\end{align}

This implies that  
\begin{equation}\label{lag_adj}
    \int_{Q_T}\left[(\Phi_s)_t  + D\Delta \Phi_s - [D_yF({y^\circ},{u^\circ})]^\top \Phi_s - \sum\limits_{i=1}^{m_s}\lambda_ip_i(y)^{p_s^i-1}\vec{a_i}^\top\right]h dQ_T = 0  
\end{equation} 
for all $h \in [C_0^\infty(Q_T)]^6,$ but since $C_0^\infty(\Omega)$ is dense in $L^p(\Omega)$ for all $p \geqslant 1$, we have 
\begin{align}\label{ad_nbni}
    -(\Phi_s)_t - D\Delta \Phi_s = [D_yF({y^\circ},{u^\circ})]^\top \Phi_s + \sum\limits_{i=1}^{m_s}\lambda_ip_i(y)^{p_s^i-1}\vec{a_i}^\top = 0
\end{align} 
in $C_0^\infty((0,T);\CB{[L^p(\Omega)]^6})$ for any desirable $p.$ Now we see that a good part of the expression of $D_y\mathcal{L}({y^\circ},{u^\circ},\Phi)h$ vanishes and on the equation $D_y\mathcal{L}({y^\circ},{u^\circ},\Phi)h = 0$ we are left with
\begin{align}
    0 &= \int_{\Sigma_T} D[(\nabla h \cdot \mathbf{n})\Phi_s - (\nabla \Phi_s \cdot \mathbf{n})h]d\Sigma_T\nonumber  - \int_{\Sigma_T} \Phi_\partial \cdot (\nabla h \cdot \mathbf{n})d\Sigma_T. 
\end{align} 
Now, since the map $h \to (h_\Gamma, (\nabla h \cdot \mathbf{n})\rvert_\Gamma)$ is surjective from the $\CB{[W^{2,p}(\Omega)]^6} \to \CB{[W^{3/2,p}(\Gamma)]^6} \times \CB{[W^{1/2,p}(\Gamma)]^6}$, if we fix $h_\Gamma = 0$ we would have $$\int_{\Sigma_T} (D\Phi_s - \Phi_\partial)\nabla h \cdot \mathbf{n} d\Sigma_T = 0, \qquad h \in \CB{[W^{2,p}(\Omega)]^6}$$ which is only possible if $D\Phi_s = \Phi_\partial$ on $\Sigma_T.$ Now, flipping the argument we get $$\int_{\Sigma_T} (D\nabla \Phi_s \cdot \mathbf{n})h d\Sigma_T = 0, \qquad h \in \CB{[W^{2,p}(\Omega)]^6}$$ whereby $\nabla \Phi_s \cdot \mathbf{n} = 0$ on $\Sigma_T.$ This supply \eqref{ad_nbni} with boundary condition, i.e., we now have 
\begin{align}\label{ad_ni}\begin{split}
    &-{\Phi_s}_t - D\Delta \Phi_s = [D_yF({y^\circ},{u^\circ})]^\top \Phi_s + \sum\limits_{i=1}^{m_s}\lambda_ip_i(y)^{p_s^i-1}\vec{a_i}^\top, \\
    &\nabla \Phi_s \cdot \mathbf{n} = 0\end{split}
\end{align} 
in $C_0^\infty((0,T);\CB{[L^p(\Omega)]^6})$ and now we only need initial (or in this case, terminal) condition. For this we use the fact that $C_0^\infty((0,T);\CB{[L^p(\Omega)]^6})$ is dense in $H^1(0,T;\CB{[L^p(\Omega)]^6})$, which means that on the equation $D_y\mathcal{L}({y^\circ},{u^\circ},\Phi)h = 0$ we are left with 
\begin{align}
    0 &= -\int_\Omega \{\Phi_s(T)h(T)-\Phi_s(0)h(0)\}d\Omega \nonumber - \int_\Omega \Phi_0 \cdot h(0) d\Omega
\end{align} 
Now again because the operator $h \mapsto (h(T),h(0))$ is surjective from $H^1(0,T;\CB{[L^p(\Omega)]^6})$ to $\CB{[L^p(\Omega)]^6} \times \CB{[L^p(\Omega)]^6}$ we first assume $h(T) = 0$ to get $$\int_\Omega (\Phi_s(0)-\Phi_0)h(0)d\Omega = 0$$ for all $h(0) \in \CB{[L^p(\Omega)]^6}$ which implies $\Phi_s(0) = \Phi_0$ a.e. in $\Omega.$ We are then left with $$0 = \int_\Omega \Phi_s(T)h(T)d\Omega = 0$$ for all $h(T) \in \CB{[L^p(\Omega)]^6}$ which implies $\Phi_s(T) = 0.$ This completes the formal computations we needed. We make it rigorous by introducing topology. We use here the minimal topology required to justify the computations. Notice that the only needed Lagrange multiplier is $\Phi_s$ because $\Phi_\partial = \Phi_s\rvert_{\Sigma_t}$ and $\Phi_0 = \Phi_s(0).$ 
\begin{definition}
    The Lagrangian function $\mathcal{L}: [W^{1,p}(Q_T)]^6 \times U_{ad}^{\CB{p}} \times \bigoplus\limits_{i=1}^{m_s}[W^{1,{p_s^i}'}(Q_T)]^6 \to \mathbb{R}$ for our control problem is defined as \begin{align}
    \mathcal{L}(y,u,\Phi) &= \mathcal{J}(y,u)  - \int_{Q_T} \Phi\cdot y_t dQ_T - \int_{Q_T} \sqrt{D}\nabla \Phi \cdot \sqrt{D} y dQ_T + \int_{Q_T}\Phi \cdot F(y,u)dQ_T,
\end{align} where \CB{the adjoint state} $\Phi = \Phi(y,u) \in \bigoplus\limits_{i=1}^{m_s}\CB{[W^{1,{p_s^i}'}(Q_T)]^6}$ is the solution to the \CB{adjoint} system
\begin{align}\label{ad_bi}\begin{split}
    &-{\Phi}_t - D\Delta \Phi = [D_yF(y,u)]^\top \Phi + \sum\limits_{i=1}^{m_s}\lambda_ip_i(y)^{p_s^i-1}\vec{a_i}^\top, \\
    &\nabla \Phi \cdot \mathbf{n} = 0,\\
    &\Phi(\cdot,T) = 0.\end{split}
\end{align} 
\end{definition} 



We end the section with the abstract optimality system provided by the Lagrangian. \begin{theorem}[\bf First order optimality system] Let ${u^\circ}$ be an optimal solution to \eqref{e.w09181_cpyu} and ${y^\circ} = \cM({u^\circ}),$ and $\Phi^\circ = \Phi({y^\circ},{u^\circ}).$  Then the following optimality system holds: 
\begin{equation}\label{optim_sys_lag}
    \begin{split} 
    &{u^\circ} \in U_{ad}^{\CB{p}}, \\
    &\mathcal{L}_{\Phi}({y^\circ},{u^\circ},\Phi^\circ) = 0, \\
    &\mathcal{L}_{y}({y^\circ},{u^\circ},\Phi^\circ) = 0, \\
    &\langle \mathcal{L}_u({y^\circ},{u^\circ},\Phi^\circ),u-{u^\circ}\rangle \geqslant 0, \qquad \forall u \in U\CB{_{ad}^p}.
    \end{split}
\end{equation}
\end{theorem} 

In the upcoming sections we will restate the optimality system \eqref{optim_sys_lag} in a more explicit way so it can serve as basis to the simulations of the last section.

\subsubsection{Fréchet differentiability of the control-to-state map}

We now return to system \begin{equation}\label{VEC_SYS2}
	\begin{split} 
		&y_t - D\Delta y = F(y,u), \\
		&y(\cdot,0) = y_0 \\
		&\nabla y \cdot \mathbf{n} = 0
	\end{split}
\end{equation}  from where we have the map $u \mapsto y = \cM(u)$ well-defined in some function spaces. We now show that $\cM$ is Fréchet differentiable (in the appropriate spaces).

\begin{theorem}\label{dif_teo}
    Let $\hat u \in U_{ad}^{\CB{p}} \cap \hat U$ where $\hat U$ is some open neighborhood of $\hat u$ in $U\CB{^p}$ and let $\hat y = \cM(\hat u)$. The map $\cM: \hat U \to \CB{[L^{p/2}(Q_T)]^6}$ is Fréchet differentiable. The directional derivative at $\hat u$ in the direction $h \in U\CB{^p}$ is given by $$\cM'(\hat u) h = y$$ where $y$ solves the linear PDE \begin{equation}\label{VEC_SYS2der}
    \begin{split} 
    &y_t - D\Delta y = F_u(\cM(\hat u), \hat u)h + \CB{F_y(\cM(\hat u), \hat u)y}, \\
    &y(\cdot,0) = 0 , \nabla y \cdot \mathbf{n} = 0
    \end{split}
\end{equation} \end{theorem}

\begin{proof} We have to show that \begin{equation}
    \label{def_dif} \cM(\hat u+h) - \cM(\hat u) = Th + r(\hat u, h)
\end{equation} where $T$ is a continuous linear operator from $\hat U$ to $[L^{p/2}(Q_T)]^6$ and $$\dfrac{\|r(\hat u, h)\|_{[L^{p/2}(Q_T)]^6}}{\|h\|_{U^p}} \to 0 $$ as $\|h\|_{U^p} \to 0.$ 

Let $\overline y = \cM(\hat u + h)$ and $\hat y = \cM(\hat u)$. The difference $\hat z = \overline y - \hat y - y$ solves the PDE \begin{equation} \label{hatz_eq}
    \begin{split} 
    &\hat z_t - D\Delta \hat z = F(\overline y,\hat u + h)-F(\hat y, \hat u) - F_u(\hat y, \hat u)h - F_y(\hat y, \hat u)y, \\
    & \hat z(\cdot,0) = 0 , \nabla 
    \hat z\cdot \mathbf{n} = 0.
    \end{split}
\end{equation} 
Since $y \mapsto F(y,u)$ is Frechét differentiable in the $L^\infty$ topology and $u \mapsto F(y,u)$ is Frechét differentiable from $L^{p}$ to $L^{p/2}$ we have
 \begin{align}F(\overline y,\hat u + h)-F(\hat y, \hat u)  &= F(\overline y,\hat u + h)-F(\overline y,\hat u) + F(\overline y,\hat u) - F(\hat y, \hat u) \\ &= F_u(\hat y, \hat u)h + r_u + F_y(\hat y,\hat u)(\overline{y}-\hat y) + r_y
\end{align}
with
\begin{equation}
    \dfrac{\|r_y\|_{\CB{[L^\infty(Q_T)]^6}}}{\|\overline y - \hat y\|_{\CB{[L^\infty(Q_T)]^6}}} \to 0 \ \mbox{as} \ \|\overline y - \hat y\|_{\CB{[L^\infty(Q_T)]^6}} \to 0
\end{equation} and \begin{equation}
    \dfrac{\|r_u\|_{[L^{p/2}(Q_T)]^6}}{\|h\|_{[L^{p}(Q_T)]^3}} \to 0 \ \mbox{as} \ \|h\|_{[L^{p}(Q_T)]^3} \to 0.
\end{equation} 
Hence \eqref{hatz_eq} becomes \begin{equation}\label{VEC_SYS2new22}
    \begin{split} 
    &\hat z_t - D\Delta \hat z + F_y(\hat y, \hat u)\hat z = r_u + r_y, \\
    &\hat z(\cdot,0) = 0 , \nabla \hat z \cdot \mathbf{n} = 0.
    \end{split}
\end{equation} which has a unique solution. Recall now that $u \mapsto \cM(u)$ is a Lipschitz continuous mapping from $U_{ad}^{\CB{p}}$ to \CB{$Y^6$} where $$ Y \equiv W^{1,p}(0,T;\CB{L^p(\Omega)}) \cap L^p(0,T;\CB{W^{2,p}(\Omega)}) \cap C([0,T];\CB{W^{1,p}(\Omega)})$$ and that the embedding $$Y \hookrightarrow C([0,T];C(\overline{\Omega}) \cap W^{1,p}(\Omega))
$$ is continuous. Hence, \begin{align}
    \|\overline y - \hat y\|_{\CB{[C(\overline{\Omega_T})]^6}} + \|\overline y - \hat y\|_{\CB{[C([0,T];W^{1,p}(\Omega)]^6)}} \leqslant C\|h\|_{\CB{U^p}}.
\end{align} \CB{Now notice that \begin{align}
    \dfrac{\|r_y\|_{[L^{p/2}(Q_T)]^6}}{\|h\|_{U^p
    }} &\leqslant \dfrac{\|r_y\|_{[L^\infty(Q_T)]^6}\|\overline{y}-\hat y\|_{[L^\infty(Q_T)]^6}}{\|\overline{y}-\hat y\|_{[L^\infty(Q_T)]^6}\|h\|_{U^p
    }}\leqslant C\dfrac{\|r_y\|_{[L^\infty(Q_T)]^6}}{\|\overline{y}-\hat y\|_{[L^\infty(Q_T)]^6}}
\end{align}
which implies that the right hand side of \eqref{VEC_SYS2new22} is $o(\|h\|_{U^p})$. The standard parabolic inequality \cite[Theorem 9.1, p. 341]{lady} then implies $\|\hat z\|_{[L^{p/2}(\Omega_T)]^6} = o(\|h\|_{U^p})$, whence the result follows from the definition of $\hat z$ and $y.$}\end{proof}

\subsubsection{The gradient of the cost function and explicit stationary system}

We now compute the gradient of the cost function under Assumption \ref{ass_pol}.


\begin{lem} Let $J(u) = \mathcal{J}(\cM(u),u)$ for $u \in U_{ad}^{\CB{p}}$. The gradient $\nabla J(u) \in (U^p)^*$ is given by \CB{\begin{equation}
    \label{gradJ} (\nabla J(u),h)_{(U^p)^*, U^p} = \int^T_0 \int_{\Omega} \Phi^\top F_u(\cM(u),u)hdQ_T + \int^T_0 \int_{\Omega}\sum\limits_{i = 1}^{m_c} \zeta_iq_i(u)\vec{c}_i^\top hdQ_T 
\end{equation}} \CB{for any $h \in U^p$}, where $\Phi = \Phi(\cM(u),u)$ is the adjoint state defined in \eqref{ad_bi}.
\end{lem}\begin{proof} Let $h \in U^p_{ad}$ fixed and put $w = \cM'(u)h.$ From \eqref{ad_bi} we have \begin{align}
    &\lim\limits_{\varepsilon \to 0^+} \dfrac{1}{\varepsilon}\int^T_0 \int_{\Omega} \sum\limits_{i = 1}^{m_s} \dfrac{\lambda_i}{p_s^i} \{p_i(\cM(u+\varepsilon h))^{p_s^i}-p_i(\cM(u))^{p_s^i}\}dQ_T \\ &= \int^T_0 \int_{\Omega} \sum\limits_{i = 1}^{m_s}{\lambda_i} p_i(\cM(u))^{p_s^i-1}\nabla p(\cM(u))^\top\cdot wdQ_T \\ &=
  \int^T_0 \int_{\Omega} \left\{-{\Phi}_t - D\Delta \Phi - [F_y(\cM(u),u)]^*\Phi\right\}^\top wdQ_T \\ &= \int^T_0 \int_{\Omega} \Phi^\top\left\{w_t - D\Delta w - [F_y(\cM(u),u)]w\right\}dQ_T \\ &= \int^T_0 \int_{\Omega} \Phi^\top F_u(\cM(u),u)hdQ_T 
\end{align} whereby the formula follows only by adding the control part of the objective functional. \end{proof}

The characterization of the gradient \eqref{gradJ} along with the ajoint state $\Phi$ allow us to write the stationary system \eqref{optim_sys_lag} in an explicit way that can be used for numerical computation of the optimal controls.

\begin{theorem}[\bf First order optimality system (adjoint approach)] Let ${u^\circ}$ be an optimal solution to \eqref{e.w09181_cpyu}, ${y^\circ} = \cM({u^\circ}),$ $\Phi^\circ = \Phi({y^\circ},{u^\circ}).$ Then the following optimality system holds: 
\begin{equation}\label{optim_sys}
    \begin{split} 
    {u^\circ} &\in U_{ad}^{\CB{p}}; \\
    y_t^\circ - D\Delta {y^\circ} &= F({y^\circ},{u^\circ}), \\
    -{\Phi_t^\circ} - D\Delta \Phi^\circ &=  [D_yF({y^\circ},{u^\circ})]^\top \Phi + \sum\limits_{i=1}^{m_s}\lambda_ip_i(y)^{p_s^i-1}\vec{a_i}^\top;\\
    {y^\circ}(\cdot,0) &= y_0;\\
    \Phi^\circ(\cdot,T) &= 0;\\
    \nabla {y^\circ} \cdot \mathbf{n} &= 0;\\
    \nabla \Phi^\circ \cdot \mathbf{n} &= 0;\\
    \CB{(\nabla J({u^\circ}), u-{u^\circ})_{(U^p)^*, U^p}} & \geqslant 0, \qquad \forall u \in \CB{U_{ad}^p}.
    \end{split}
\end{equation} 
\end{theorem} \CB{Here \eqref{optim_sys} represents the coupled system solved by the optimal state $y^\circ$ and adjoint state $\Phi^\circ$.}

\section{Simulations and Discussion} \label{sec:simulation}

As mentioned earliner in Section \ref{eoc}, for simulations we will use \begin{align}\label{sim_}
    \mathcal{J}(y,u) 
    &= \sum\limits_{i = 1}^{m_s} \dfrac{\lambda_i}{p_s^i}\|p_i(y)\|_{L^{p_s^i}(Q_T)}^{p_s^i} + \sum\limits_{i = 1}^{m_c} \dfrac{\zeta_i}{2}\|q_i(u)\|_{L^{2}(Q_T)}^{2}
\end{align}
with $m_s = 2$,  $p_s = (1,1), m_c = 3$ and the polynomials $p_1(\vec{x}) = x_2 + x_5, p_2(\vec{x}) = x_4 + x_5 + x_6$, $q_1(\vec{x}) = x_1$, $q_2(\vec{x}) = x_2,$ $q_3(\vec{x}) = x_3$ with $\zeta_i = \zeta$ ($i = 1,2,3$). Recalling that $y = (S,I,R,S^*,I^*,R^*)$ and $u = (\alpha, \mu, \nu)$ we have explicitly \begin{align}\label{sim_1}
    \mathcal{J}(y,u) &= \int^T_0 \int_{\Omega} [\lambda_1(I+I^*) + \lambda_2(S^* + I^* + R^*)]dxdt + \dfrac{\zeta}{2} \int_0^T \int_\Omega (\alpha^2 + \mu^2 + \nu^2)dxdt.
\end{align} 

We write the Lagrange multiplier $\Phi$ as $\Phi = (\Phi_S, \Phi_I, \Phi_R, \Phi_{S^*}, \Phi_{I^*}, \Phi_{R^*})$. From multi-variable calculus we know that \begin{align}
   & [F_y(y,u)]^\top = J_F(y,u)^\top
\end{align} where $J_F$ denotes the Jacobian matrix. Moreover, \begin{align}
    \sum\limits_{i=1}^{m_s}\lambda_ip_i(y)^{p_s^i-1}\vec{a_i}^\top = \begin{bmatrix}
        0 & \lambda_1 & 0& \lambda_2 & \lambda_1 + \lambda_2 & \lambda_2
    \end{bmatrix}^\top.
\end{align} The adjoint system is then explicitly given by \begin{align}\label{ad_bi22}\begin{split}
    &-{\Phi}_t - D\Delta \Phi = J_F(y,u)^\top \Phi + \begin{bmatrix}
        0 & \lambda_1 & 0& \lambda_2 & \lambda_1 + \lambda_2 & \lambda_2
    \end{bmatrix}^\top, \\
    &\nabla \Phi \cdot \mathbf{n} = 0,\\
    &\Phi(\cdot,T) = 0.\end{split}
\end{align} 
%

\CB{To approximately solve the optimization problem \eqref{e.w09181_cpyu} using the adjoint formulation, we employ a projected gradient descent  algorithm like that described in \cite[\S2.1]{Herzog}, computing the gradient of the cost functional using \eqref{gradJ}. We begin with a relatively large gradient descent rate and decrease this occasionally as the control maps begin to refine. This is fully detailed in Algorithm \ref{Algorithm1}. Referring to the parameters in the algorithm, we set $\text{TOL} = 10^{-3}, \eta = 0.1, c = 0.2, k = 10$, though other choices would likely work as well. } 

\begin{algorithm}[t!]
\caption{Projected Gradient Descent Algorithm for \eqref{e.w09181_cpyu}}
\hspace*{\algorithmicindent} Input values for all parameters included in Table \ref{tab:baselineParams}, as well as a convergence tolerance TOL, a maximum iteration count $N$, and a gradient descent rate $\eta > 0$. Also include parameters $c \in (0,1)$ and $k \in \mathbb N$ which be used to incrementally reduce the gradient descent rate as the control maps refine. Randomly initialize the control maps $u^{(0)} = (\alpha^{(0)},\mu^{(0)},\nu^{(0)})$ and state maps $y^{(0)}.$
\begin{algorithmic}
\FOR{$n = 1$ to $N$}
\STATE Compute $y^{(n)}$: the solution of \eqref{VEC_SYS} with control map $u^{(n-1)}$ 
\STATE Compute $\Phi^{(n)}$: the solution of \eqref{ad_bi22} with state map $y^{(n)}$ and control map $u^{(n-1)}$
\STATE Compute the gradient $\nabla J^{(n)}$ from \eqref{gradJ} using $y^{(n)},\Phi^{(n)},u^{(n-1)}$
\STATE Set $\tilde u = (\tilde \alpha, \tilde \mu, \tilde \nu) = u^{(n-1)} - \eta \nabla J^{(n)}.$
\STATE Set $u^{(n)} = (\min(\max(\tilde \alpha,\underline \alpha),1),\min(\max(\tilde \mu,0),\overline \mu), \min(\max(\tilde \nu,0),\overline \nu))$
\STATE Set $\text{CHANGE} =\max\Big\{\max_{i=1,\ldots,6} \|y_i^{(n)} - y_i^{(n-1)}\|_\infty ,\max_{i=1,2,3}\|u_i^{(n)} - u_i^{(n-1)}\|_\infty \Big\}$
\IF{$\text{CHANGE} < \text{TOL}$}
\STATE break loop
\ENDIF
\IF{$\text{mod}(n,k)=0$}
\STATE Reset $\eta \longleftarrow c\eta$
\ENDIF
\ENDFOR
\end{algorithmic}\label{Algorithm1}
\end{algorithm}

The authors of \cite{parkinson2023analysis} include analysis of basic reproduction numbers and provide several simulations which explore the intricacies of this model, including dependence of solutions on diffusion coefficients and different constant values of $\alpha,\mu,\nu.$ Because the novelty of this work is the optimal control of $\alpha,\mu,\nu$ with the goal of minimizing \eqref{e.w10151_gen}, we opt to fix parameter values and examine the behavior of the model and the optimal control maps on the parameters $\lambda_1,\lambda_2,\zeta$ which, respectively, serve as weights on the cost of total number of infections, total noncompliance, and $L^2$-norm of the control maps. For convenience, we break down the cost functional into these three terms: \begin{equation}
    \label{eq:costBreakdown}
    \mathcal J(y,u) = \lambda_1 I_{\text{total}} + \lambda_2 N^*_{\text{total}} + \frac \zeta 2 C_{\text{total}}
\end{equation} where 
\begin{equation}
    \label{eq:costs}
    \begin{split}
        I_{\text{total}} &= \int^T_0 \int_{\Omega} (I(x,t)+I^*(x,t))dxdt, \\
        N^*_{\text{total}} &= \int^T_0 \int_\Omega (S^*(x,t) + I^*(x,t)+R^*(x,t))dxdt, \\ 
        C_{\text{total}} &= \int^T_0\int_\Omega (\alpha(x,t)^2 + \mu(x,t)^2 + \nu(x,t)^2)dxdt.  
    \end{split}
\end{equation}

For all of our simulations, we use the two-dimensional spatial domain $\Omega = [-5,5]^2$, and a time horizon of $T = 200$. For both \eqref{PDE_SYS} and \eqref{ad_bi22}, we use a first-order approximation to time derivatives and a semi-implicit scheme for spatial derivatives where the diffusion is resolved implicitly and all other terms are treated explicitly. We include our baseline parameter values (including initial conditions) in table \ref{tab:baselineParams}. The baseline values of $\lambda_1,\lambda_2,\zeta$ were chosen so that with baseline parameter values each of the terms in \eqref{eq:costBreakdown} contributes roughly an equal amount to the total value of the cost functional. We note that in the uncontrolled case ($\alpha \equiv \underline \alpha$ and $\mu,\nu\equiv 0$), there is natural recovery from the disease whereas there is no ``recovery" from noncompliance. Accordingly, the entire population will eventually become noncompliant, while the disease will naturally die out or settle at some endemic equilibrium which does not constitute the entire population. Thus to balance the cost $\lambda_2 N^*_{\text{total}}$ associated with noncompliance and the cost $\lambda_1 I_{\text{total}}$ associated with the disease, we choose a baseline value for $\lambda_2$ which is significantly smaller than that of $\lambda_1$ (specifically, $\lambda_1 = 3, \lambda_2 = 1/50$). 

\begin{table}
\centering
    \begin{tabular}{|l c r|}
    \hline
        Parameter & Description & Value\\
    \hline
        $b(x,y)$ & natural birth rate & $0.1e^{-(x^2+y^2)}$\\
        $\xi$ & portion of newly introduced pop. which is compliant & 1\\
        $\beta$ & infection rate of disease & 5\\
        $\gamma$ & recovery rate & 1\\
        $\delta$ & natural death rate & 0.001 \\
        $\overline \mu$ & ``infection" rate of noncompliance & 1\\
        $\overline \nu$ & maximal rate of ``recovery" from noncompliance & 1\\
        $\underline \alpha$ & natural reduction in infectivity due to compliance & 0.1\\
        $d$ & mutual value of all diffusion coefficients & $0.1$ \\
        $\lambda_1$ & weight on cost of total infections & 3\\
        $\lambda_2$ & weight on cost of total noncompliance & 0.02\\
        $\zeta$ & weight on cost of control use & 0.2\\
        $S_0(x,y)$ & initial compliant susceptible population density & $e^{-(x^2+y^2)}$ \\
        $I_0(x,y)$ & initial compliant infectious population density & $0.1S_0(x,y)$ \\
        $S^*_0(x,y)$ & initial noncompliant susceptible population density & $0.05S_0(x,y)$ \\
        $I^*_0(x,y)$ & initial noncompliant infectious population density & $0.05I_0(x,y)$ \\
        \hline
    \end{tabular}
    \caption{\CB{Baseline parameter values for the simulation in figure \ref{fig:1}. We vary the cost weights $\zeta, \lambda_1,\lambda_2$ in figures \ref{fig:3}, \ref{fig:4}, and \ref{fig:5}}.}
    \label{tab:baselineParams}
\end{table}

We emphasize that the values of total infection cost, total noncompliance cost, and total control cost listed in \eqref{eq:costs} are only constant once $\lambda_1,\lambda_2,\zeta$ are fixed. Changing the cost weights will change the optimal control maps, which will give rise to different dynamics and thus change the values of \eqref{eq:costs}. \CB{In fact, exploring how the costs change as functions of the weights will be of particular interest for us.} One quantity of special interest will be the \emph{relative cost reduction} which is achieved by using the control maps when compared with the uncontrolled case. For fixed values of weights $\lambda_1,\lambda_2,\zeta$, this can be defined by \begin{equation} \label{eq:costReduction} \text{RelCR}(\lambda_1,\lambda_2,\zeta) = \frac{\mathcal J(y,\underline \alpha,0,0)-\mathcal J({y^\circ},\alpha^\circ,\mu^\circ,\nu^\circ)}{\mathcal J(y,\underline \alpha,0,0)} \in [0,1], \end{equation} where $\alpha^\circ, \mu^\circ,\nu^\circ$ are the optimal control maps associated with the weights $\lambda_1,\lambda_2,\zeta.$ In all cases, we will have $\mathcal J({y^\circ},\alpha^\circ,\mu^\circ,\nu^\circ) \le \mathcal J(y,\underline \alpha,0,0)$ since $\alpha^\circ,\mu^\circ,\nu^\circ$ achieve the infimum value. A scenario where $\mathcal J({y^\circ},\alpha^\circ,\mu^\circ,\nu^\circ) = \mathcal J(y,\underline \alpha,0,0)$ would yield $\text{RelCR}=0$ indicating that there is no reduction in cost due to the use of controls (or in other words, a scenario where the optimal plan is to neglect the use of controls entirely). By contrast, a scenario where $\mathcal J(y^\circ,\alpha^\circ,\mu^\circ,\nu^\circ) \approx 0$ would yield $\text{RelCR} \approx 1$, indicating that the use of controls can almost entirely eliminate cost. In practice, we will multiply $\text{RelCR}$ by $100$ so as to report the result as a percentage reduction in cost.  

One final note---before we present the results of simulations---is that all of these parameter values are entirely synthetic and do not represent a high-fidelity effort to model any given epidemic in any given region. Rather, we are interested in drawing qualitative conclusions regarding our optimal control problem and the manners in which a policy-maker's intervention can affect disease spread when there is a noncompliant portion of the population.

In figure \ref{fig:1}, we simulate the model with the baseline value of parameters (table \ref{tab:baselineParams}) in the absence (top) and presence (bottom) of controls. In these images (as in several ensuing images), we display the following quantities as functions of time, beginning from top left panel: \begin{itemize}
\item[(1)] Total susceptible population: $\|S(\cdot,t)+S^*(\cdot,t)\|_{L^1(\Omega)}$
\item[(2)] Total infected population: $\|I(\cdot,t)+I^*(\cdot,t)\|_{L^1(\Omega)}$
\item[(3)] Total control of the infection: $\|\alpha(\cdot,t)\|_{L^1(\Omega)}$
\item[(4)] Total compliant population: $\|S(\cdot,t)+I(\cdot,t)+R(\cdot,t)\|_{L^1(\Omega)}$
\item[(5)] Total noncompliant population: $\|S^*(\cdot,t)+I^*(\cdot,t)+R^*(\cdot,t)\|_{L^1(\Omega)}$
\item[(6)] Total control of compliance: $\|\mu(\cdot,t)\|_{L^1(\Omega)}$ and $\|\nu(\cdot,t)\|_{L^1(\Omega)}$
\end{itemize} All the populations have been normalized by dividing by the total population (which changes slightly in time due to births $b(x,y)$ and deaths $\delta$), so all these represent percentages of the total population. 

\begin{figure}[!ht]
\centering 
\includegraphics[width = \textwidth]{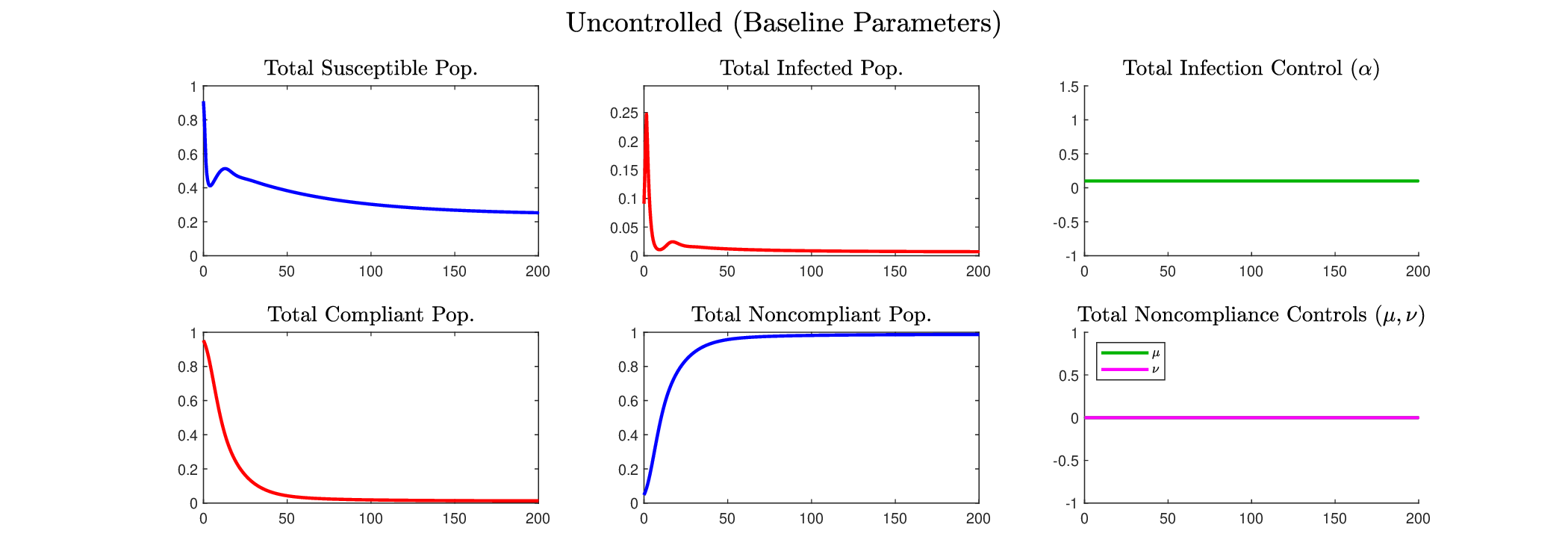}\\

\vspace{0.2cm}
\rule{0.8\textwidth}{1pt}
\vspace{0.2cm}

\includegraphics[width = \textwidth]{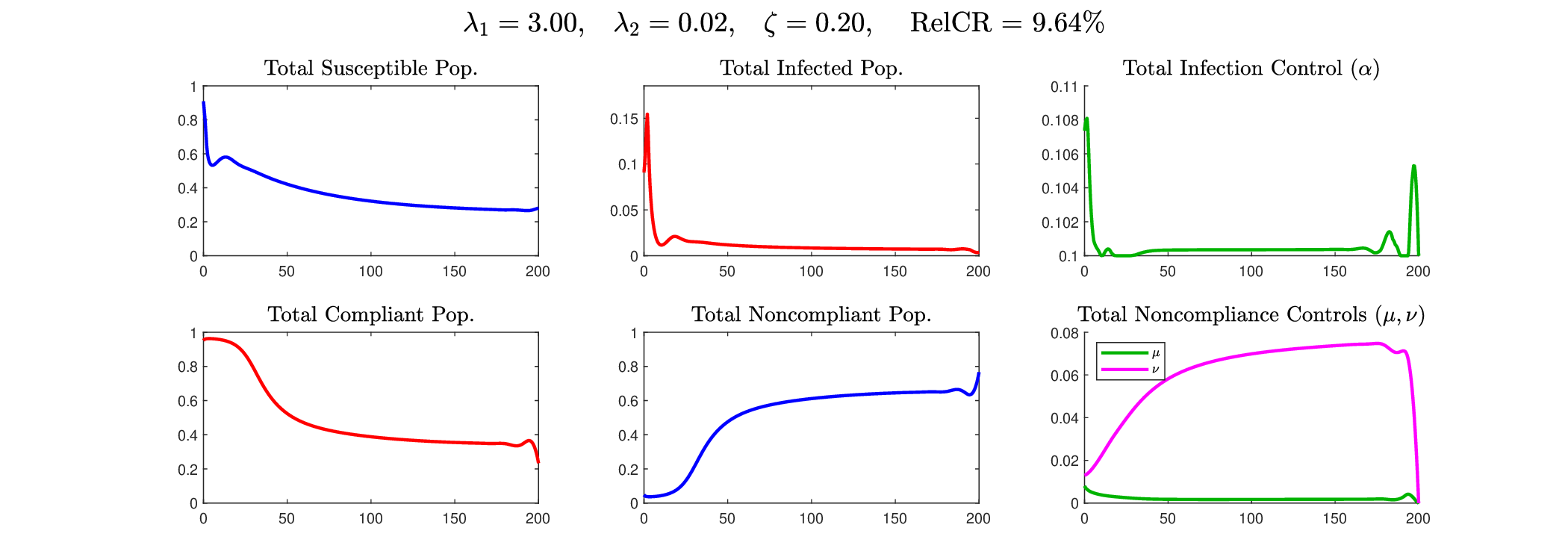}
    \caption{Dynamics of the model with baseline parameters (table \ref{tab:baselineParams}) in the absence (top) and presence (bottom) of controls. We notice that in the controlled case, the optimal $\alpha(\cdot,t)$ is primarily used near the beginning of the dynamics to decrease the first wave of infections. The optimal $\mu(\cdot,t)$ is hardly used at all, and the optimal $\nu(\cdot,t)$ is used throughout. This has the effect that the total noncompliant population settles at a lower portion of the population. The variation in the controls as the end of the dynamics should be seen as artificial: they are there because the policy-maker is aware of the time-horizon $T = 200$, and can slightly decrease costs by drastically altering controls for the final few time steps. Overall, with these values of $\lambda_1,\lambda_2,\zeta$, the optimal controls achieve a $9.64\%$ relative cost reduction against the uncontrolled scenario, reducing the cost from $\mathcal J(y,\underline \alpha,0,0) = 1.4071$ to $\mathcal J(y,\alpha^\circ,\mu^\circ,\nu^\circ) = 1.2715$. Snapshots of the control maps $\alpha(x,t),\mu(x,t),\nu(x,t)$ at different times are displayed in figure \ref{fig:2}.}
    \label{fig:1}
\end{figure} 

\begin{figure}[!ht]
\centering 
\includegraphics[width=0.24\textwidth]{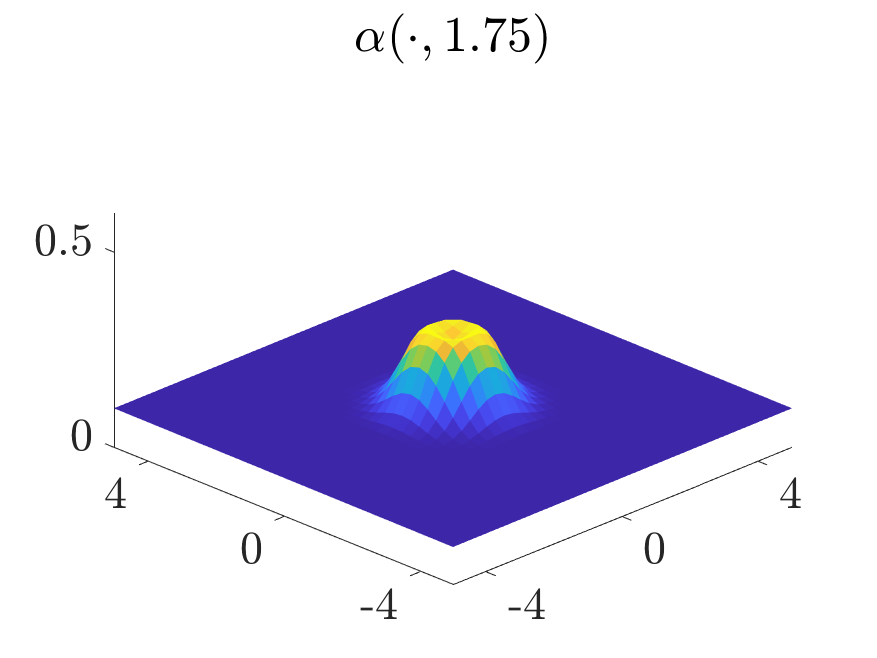}\,
\includegraphics[width=0.24\textwidth]{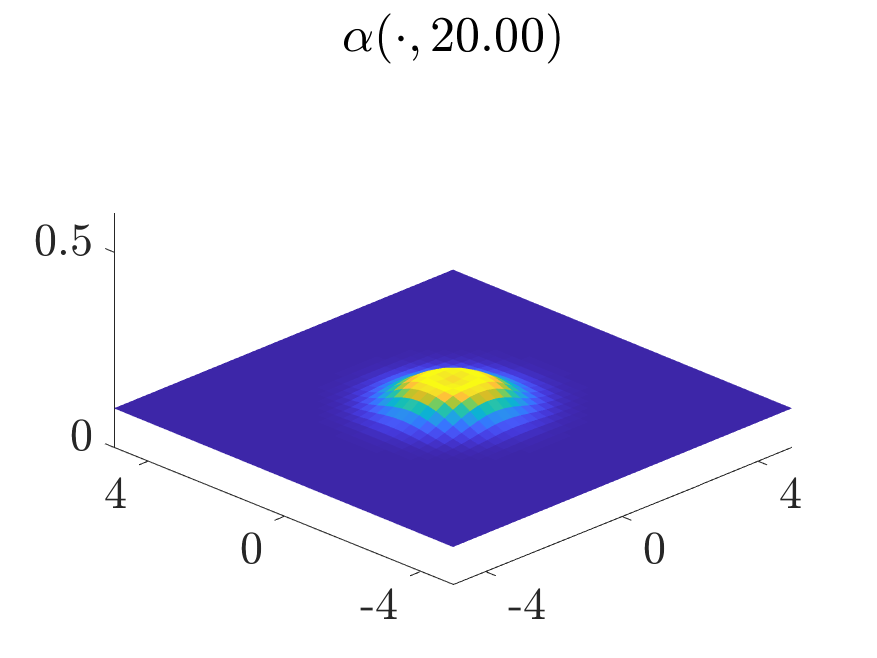}\,
\includegraphics[width=0.24\textwidth]{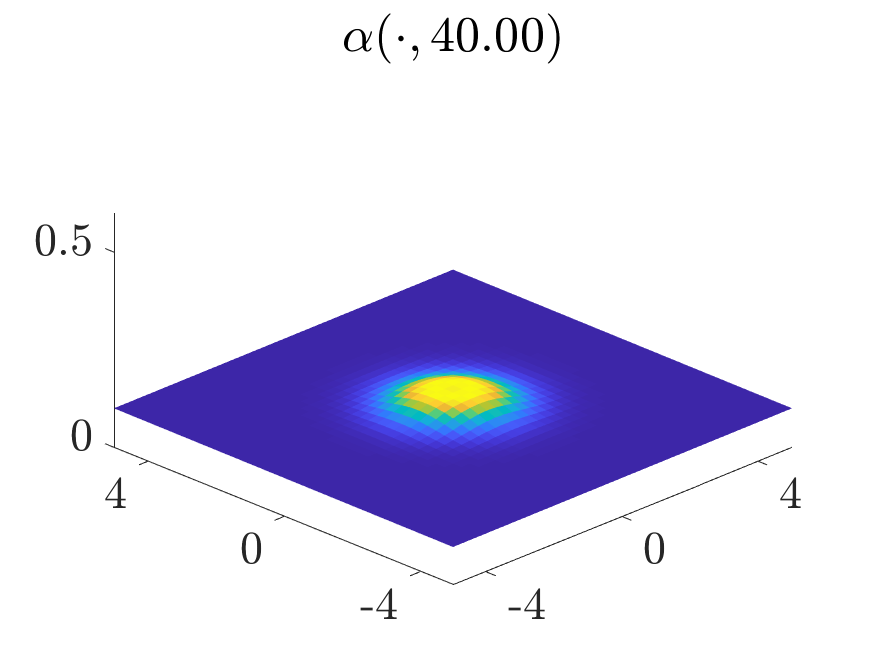}\,
\includegraphics[width=0.24\textwidth]{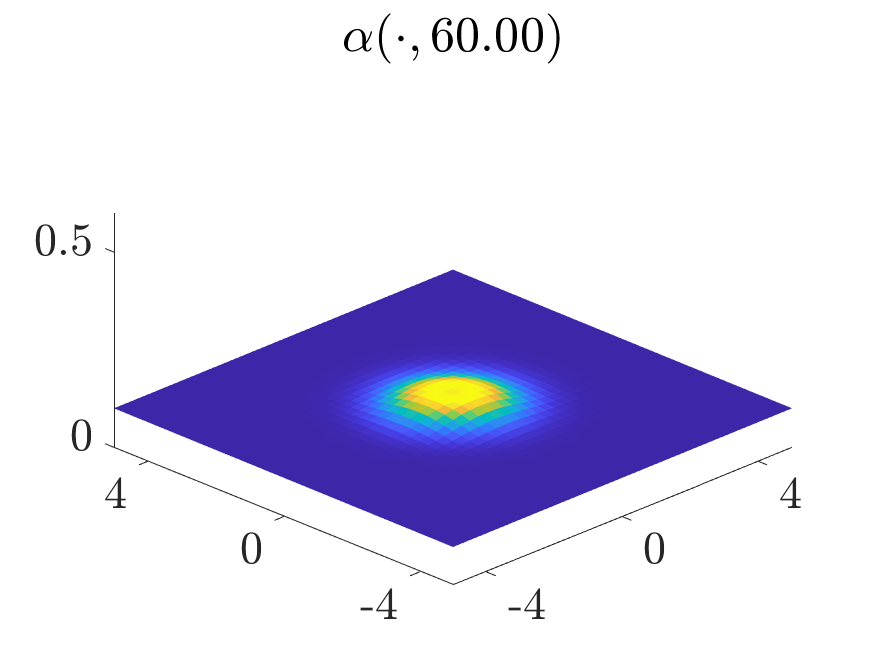}\\
\includegraphics[width=0.24\textwidth]{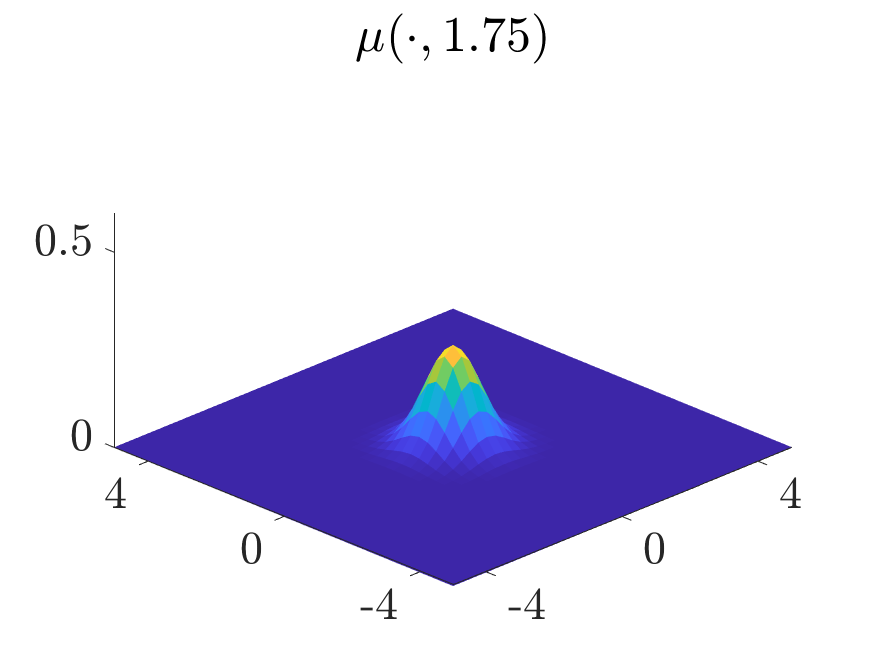}\,
\includegraphics[width=0.24\textwidth]{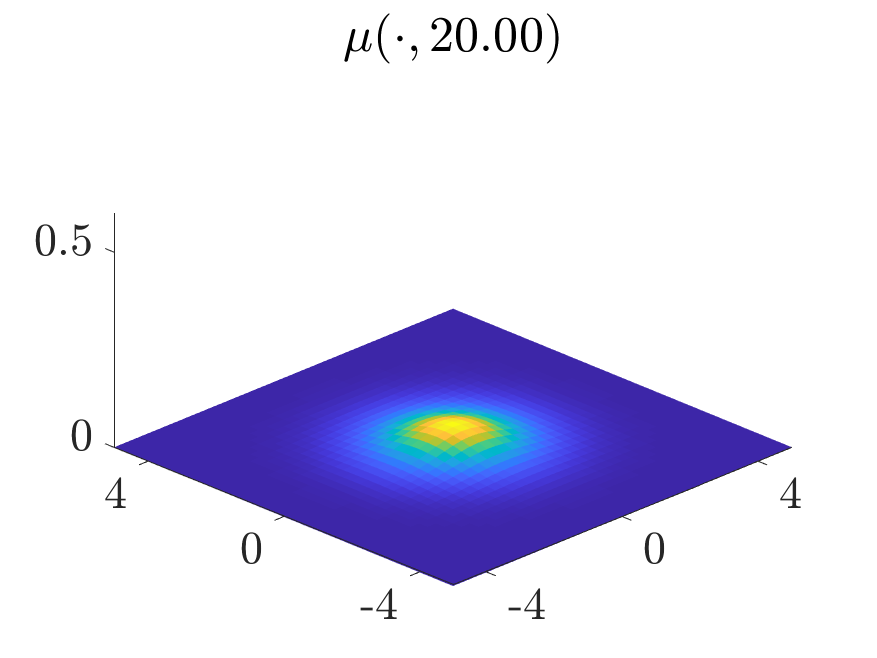}\,
\includegraphics[width=0.24\textwidth]{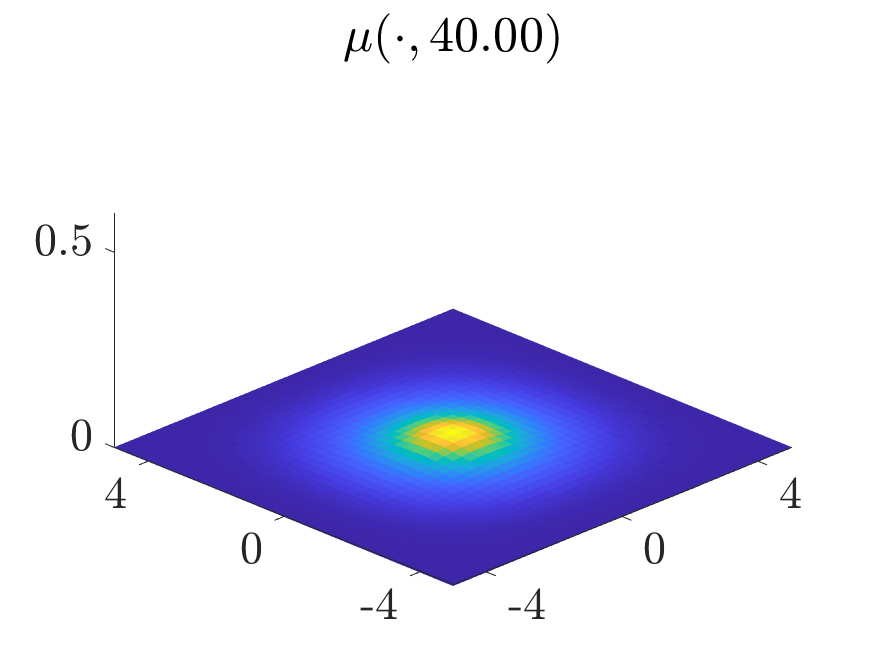}\,
\includegraphics[width=0.24\textwidth]{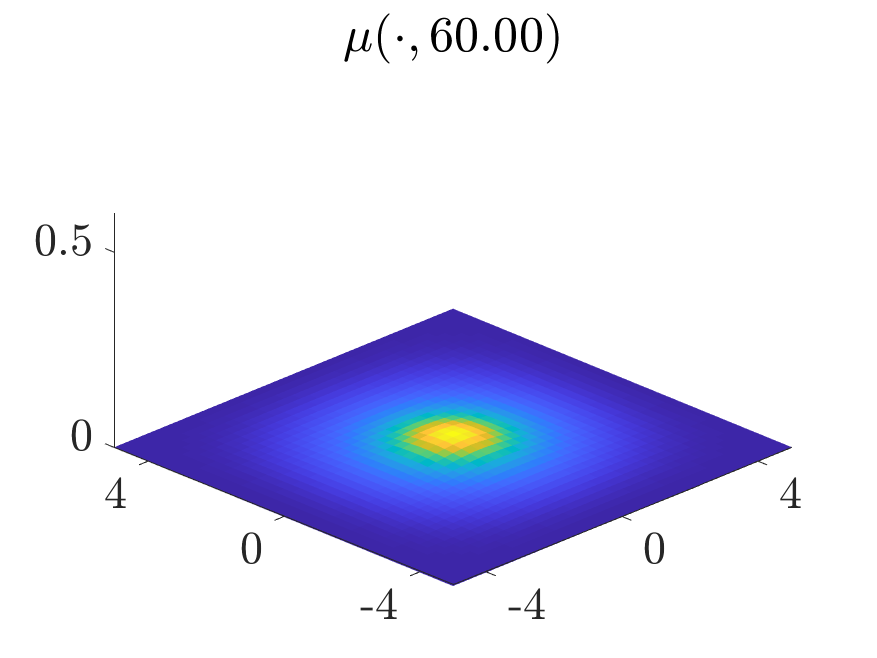}\\
\includegraphics[width=0.24\textwidth]{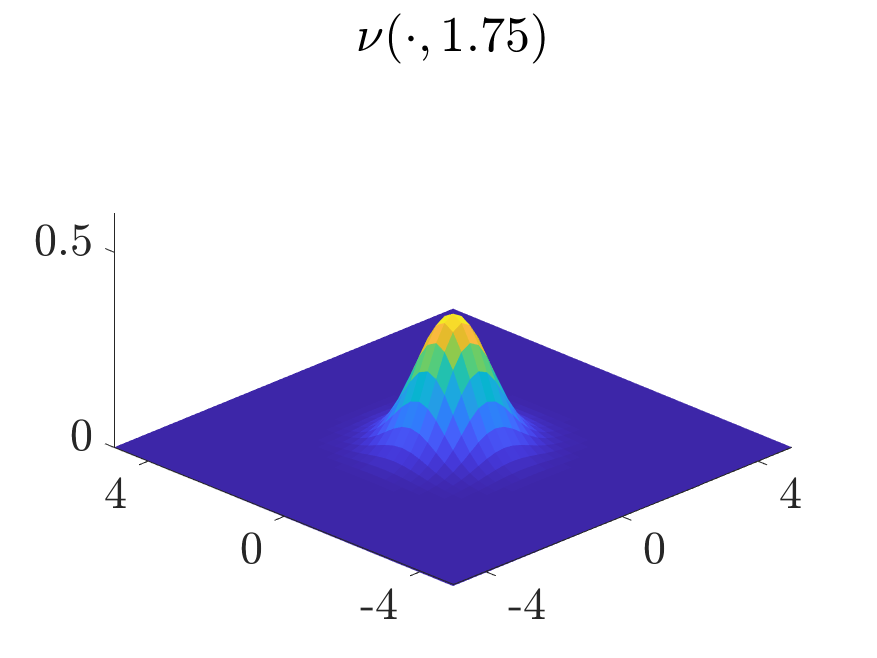}\,
\includegraphics[width=0.24\textwidth]{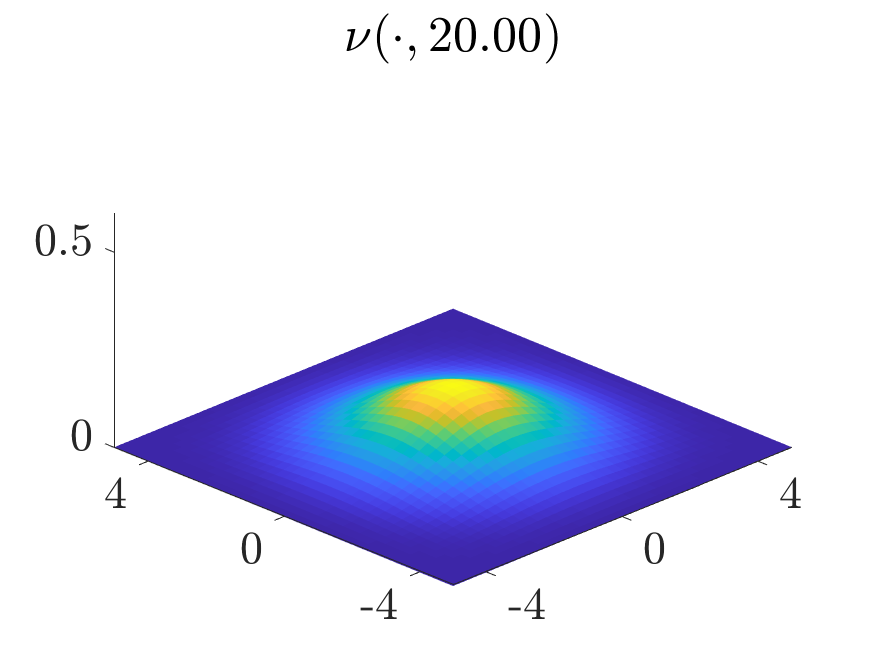}\,
\includegraphics[width=0.24\textwidth]{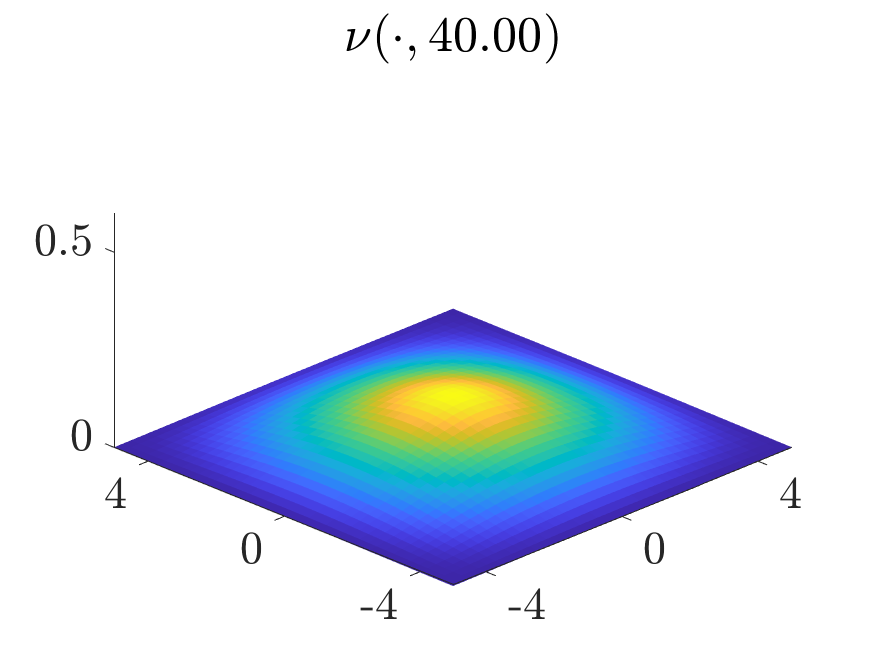}\,
\includegraphics[width=0.24\textwidth]{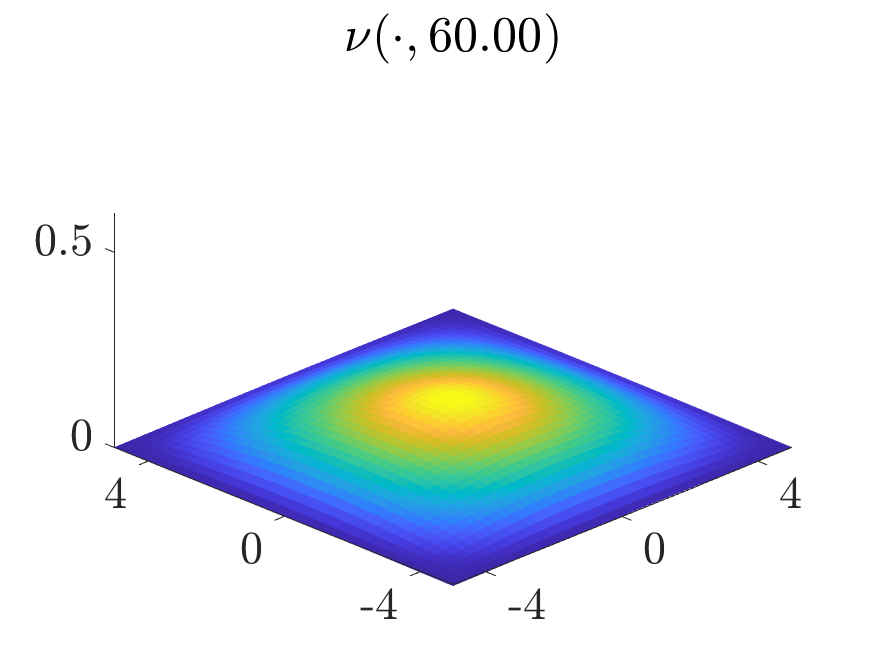}
\caption{Snapshots of the optimal control maps $\alpha(x,t),\mu(x,t),\nu(x,t)$ for different times $t$ for simulation with baseline parameters The time $t = 1.75$ corresponds to the first peak in infections seen in figure \ref{fig:1}. The control efforts are concentrated near the origin because $b(x,y)$ and $S_0(x,y)$ are Gaussians centered at the origin, meaning this is where the bulk of the population is. Note that as the infection dies out over time, $\alpha(x,t)$ and $\mu(x,t)$ seem to decrease. However, $\nu(x,t)$ decreases at its peak, but grows elsewhere. This is the primary mechanism used to decrease the final asymptote for the noncompliant population, and thus achieve a decreased total cost, despite the increase in the cost of the control.}
\label{fig:2}
\end{figure}

The affect of the control maps on the dynamics is summarized more fully in the captions, but in short, when using the baseline parameters, the optimal control strategy is to use $\alpha(\cdot,t)$ for small time to decrease the first peak in infections, thus lowering $I_{\text{total}}$. One then uses $\mu(\cdot,t)$ and (to a larger extent) $\nu(\cdot,t)$ throughout the dynamics so that the noncompliant population occupies a smaller portion of the total population, thus decreasing $N^*_{\text{total}}.$ In doing so, although some cost is being ``spent" on controls, the overall cost decreases by roughly $\text{RelCR} = 6.92\%$ from $\mathcal J(y,\underline \alpha,0,0) = 1.0041$ to $\mathcal J(y,\alpha^\circ,\mu^\circ,\nu^\circ) = 0.9346$. Figure \ref{fig:2} displays snapshots of the control maps at different times for this same simulation. The control maps are concentrated near the origin because that is where the majority of the population resides (since $b(x,y)$ and $S_0(x,y)$ are Gaussians centered at the origin). Generally, it will only be profitable to enforce controls where people are present, so the basic shape of the control maps is essentially the same in all the ensuing simulations. Accordingly, we focus on dynamics like those plotted in figure \ref{fig:1} for the ensuing simulations. 

For our next simulations, we vary $\zeta$---the cost associated with the control maps---while fixing all other parameters at their baseline values. The results of two such simulations are included in figure \ref{fig:3}. In the top panel, we take a much smaller $\zeta = 0.1$, meaning controls are much cheaper to implement. In this case, using very strong controls, one prevents the initial outbreak of the disease entirely. Spread of noncompliance is also significantly delayed, though when enough of the population becomes noncompliant a smaller and flatter outbreak of the disease occurs. In doing so, one achieves a relative cost reduction of $39.61\%$ against the uncontrolled case. Decreasing $\zeta$ further causes the outbreak to be delayed longer and flatten more, and leads to even larger relative cost reduction. In the bottom panel, we take a much larger $\zeta = 1$. Here the basic shapes of the control maps look very similar to their baseline shapes, but they are significantly scaled down since they are significantly more expensive to implement. Here, one can only achieve a relative cost reduction of $2.64\%$ against the uncontrolled case. This behavior persists upon increasing $\zeta$ further: the control maps are qualitatively similar, but scaled down further and further until one essentially reaches the uncontrolled scenario. Note that in the baseline case ($\zeta = 0.25$, figure \ref{fig:1}) there is still an initial outbreak which occurs before noncompliance spreads. This peak is no longer present when $\zeta = 0.2$, so the baseline value is very near the threshold for when it is optimal to suppress the initial outbreak.  

\begin{figure}[!ht]
    \centering 
\includegraphics[width = 0.9\textwidth]{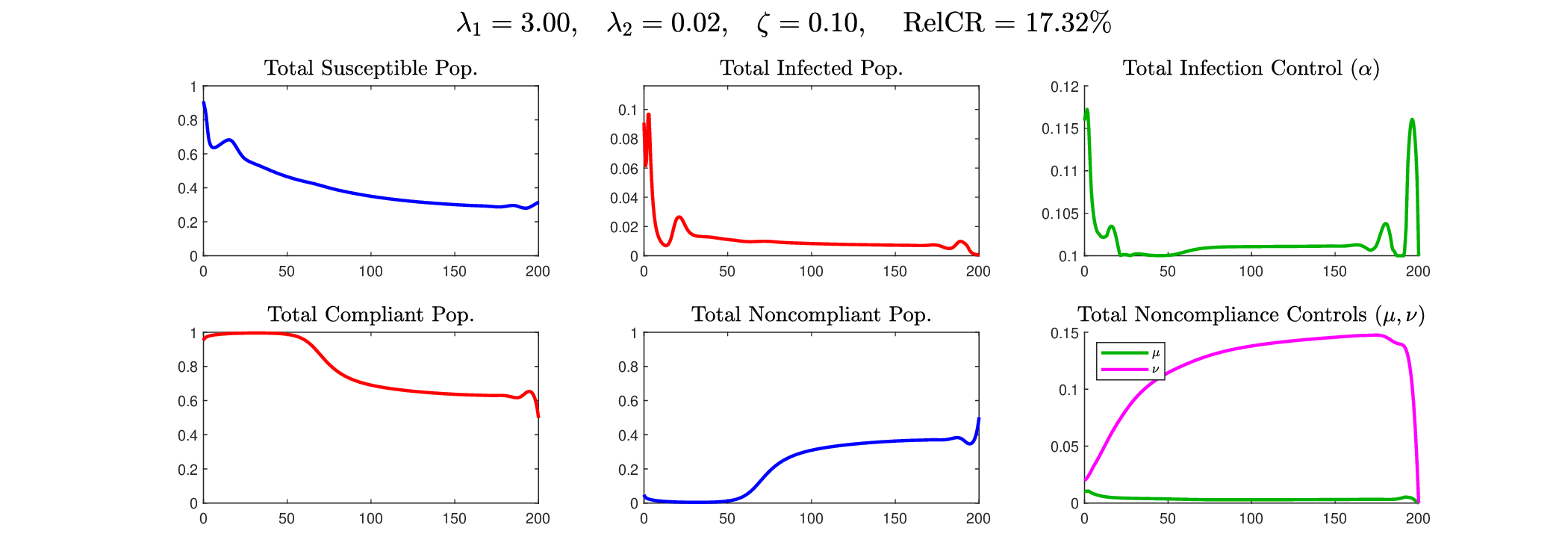}\\

\vspace{0.2cm}
\rule{0.8\textwidth}{1pt}
\vspace{0.2cm}

\includegraphics[width = 0.9\textwidth]{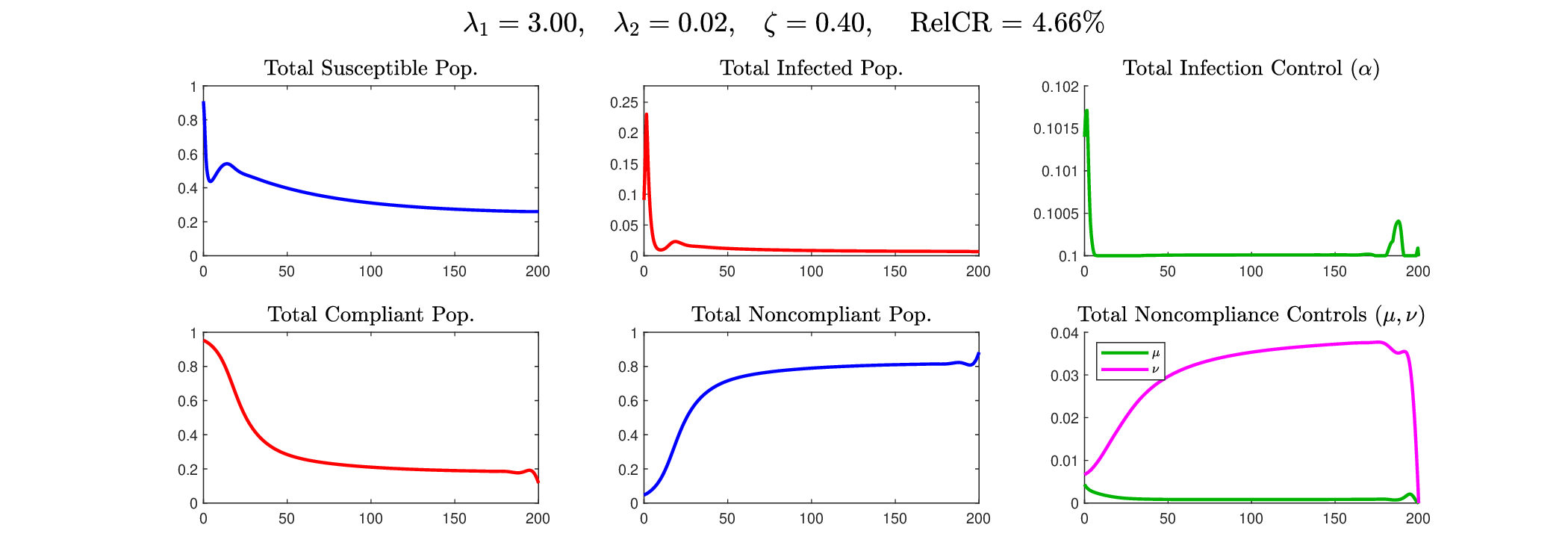}
\caption{When we decrease $\zeta$ to $0.1$ (top), the controls are cheap enough to implement that the optimal strategy is now to significantly suppress the initial outbreak, and to suppress noncompliance initially. Eventually noncompliance spreads and a small, more gradual outbreak occurs. In this case, the relative cost reduction against the uncontrolled scenario is $17.32\%$. When we increase $\zeta$ to $0.4$ (bottom), the results look qualitatively similar to the baseline case except the control maps are significantly scaled down because controls are more expensive to implement. In this case, the relative cost reduction against the uncontrolled scenario is only $4.66\%$}
\label{fig:3} 
\end{figure}

In our next two simulations, we demonstrate the affect of changing $\lambda_1$. Unsurprisingly, if all else is held equal and $\lambda_1$ is decreased (so that the policy-maker has less concern for the infection spreading), the only significant effect is that $\alpha$ is decreased. Otherwise, the results are qualitatively similar (in particular, $\mu$ and $\nu$ remain roughly the same because they are still required in order to reduce the spread of noncompliance). Rather than display these results, we focus on the case of increasing $\lambda_1$. Specifically, in figure \ref{fig:4}, we significantly increase the cost of total infections by setting $\lambda_1=10$. This would correspond to a very public-health-oriented policy-maker. In this case, if all other parameters are held at baseline values (top), the optimal strategy is to use very large $\alpha, \mu$, and $\nu$ values to suppress the outbreak entirely. Note that not only $\alpha$ is increased: in this case, it behooves one to increase $\mu$ and $\nu$ as well because in this model keeping the population compliant makes it easier to slow disease spread. Also notice that in this case $\alpha$ has very oscillatory behavior. For the sake of this simulation, we have significantly zoomed in on the graph of the total infected population: total infections immediately decrease and remain very small, but display tiny oscillations, and the oscillatory behavior in $\alpha$ is the health-conscious policy maker responding to these tiny oscillations. In figure \ref{fig:4} (bottom), we again set $\lambda_1 = 10$ but now also increase the cost of implementing controls by setting $\zeta = 1$. With these values, there is a serious trade-off to consider: the policy-maker strongly desires to stop the spread of the disease, but implementing controls is very costly. In this case, it seems suppressing the disease spread entirely is simply too costly. Instead, the infection is initially suppressed, but allowed to reach a sharp peak when noncompliant behavior becomes widespread enough. After the sharp peak, enough of the population has gained immunity that one can keep the infections relatively low despite the high cost of controls.

\begin{figure}[!ht]
    \centering 
\includegraphics[width = 0.9\textwidth]{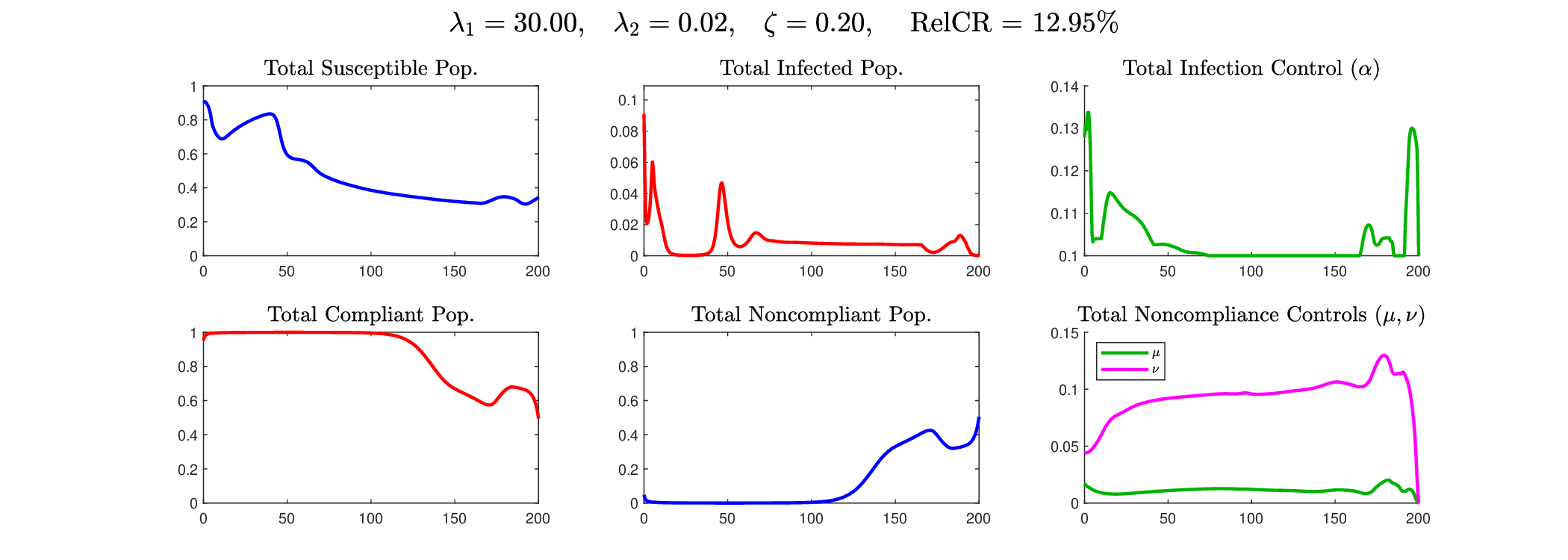}\\

\vspace{0.2cm}
\rule{0.8\textwidth}{1pt}
\vspace{0.2cm}

\includegraphics[width = 0.9\textwidth]{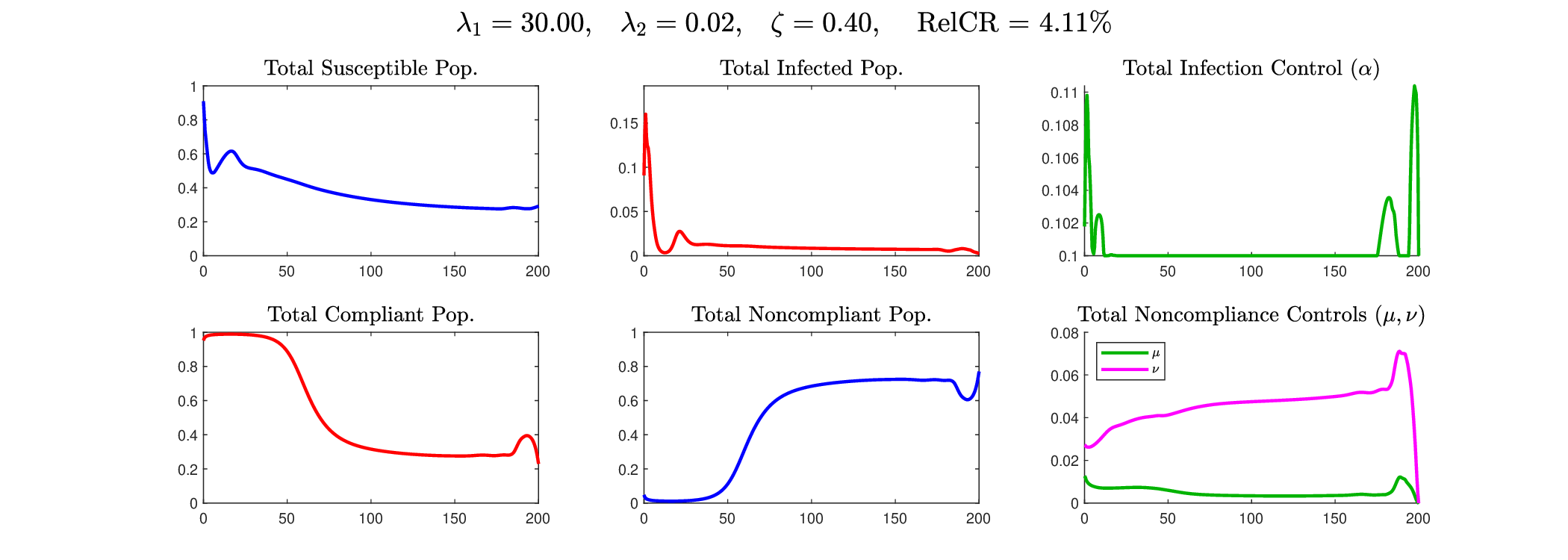}
\caption{We consider a public-health-oriented policy-maker by increasing $\lambda_1$ to $30$. In the top panel, we use baseline parameter values and the optimal policy is to use large control values to suppress the infection. In this case, the policy-maker is sensitive to small changes in the total infected population, leading to more oscillation in the control strategies. In the bottom panel, we increase $\lambda_1$ to 30 (so the policy-maker is health-conscious) but also increase $\zeta$ to 0.4 (so controls are costly to implement). Interestingly, in this case, the optimal use of $\alpha(\cdot,t)$ does not qualitatively change much from baseline (though it is larger), but noncompliance is suppressed more strongly. This demonstrates some of the synergy between the desire of the policy-maker and the use of controls: here the policy maker is health-conscious, but lowers infections by increasing control of both infections and noncompliance.}
\label{fig:4}
\end{figure}

For our last simulation, we look at the effect of changing $\lambda_2$. As with $\lambda_1$, the case of decreasing $\lambda_2$ is not particularly interesting. In this case, the policy-maker does not care to encourage compliance, so $\mu$ and $\nu$ are decreased. In turn, this makes $\alpha$ less useful since increasing $\alpha$ only provides a benefit to compliant populations. Thus when $\lambda_2$ is decreased, the dynamics qualitatively conform to the uncontrolled case in figure \ref{fig:1}. However, increasing $\lambda_2$ leads to an interesting change in the behavior when compared with the baseline parameters. In figure \ref{fig:5}, we display the results of the simulation when $\lambda_2$ is increased tenfold to $\lambda_2 = 0.2$. This would correspond to a policy-maker who is most concerned that people remain compliant. In this case, the primary strategy is increasing $\mu$ and $\nu$ so as to entirely eliminate noncompliance. However, this has the secondary effect of making $\alpha$ more useful since the whole population is compliant. Accordingly, $\alpha$ is also increased and the spread of the disease is also significantly reduced. This represents the potential for synergy between the different effects of the control variables. 

\begin{figure}[!ht]
    \centering 
\includegraphics[width = 0.9\textwidth]{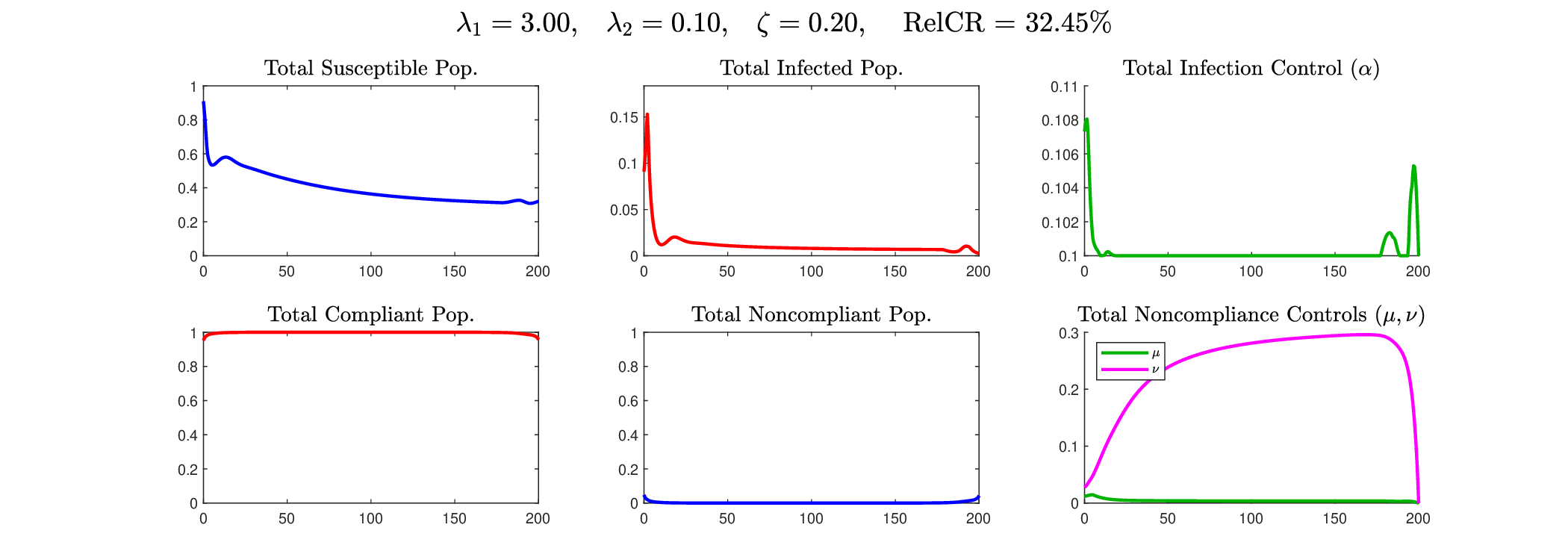}\\

\caption{We consider a compliance-oriented policy maker by increasing $\lambda_2$ fivefold to $\lambda_2 = 0.1$. In this case, the optimal strategy is to increase $\mu$ and $\nu$ so as to eliminate noncompliance entirely. However, having done so, $\alpha$ is more effective as a control since everyone is compliant, so the policy-maker still uses $\alpha$ as well, thus significantly slowing the spread of the disease. This once again demonstrates the strong potential for synergy between the control variables.}
\label{fig:5}
\end{figure}

\section{Conclusion \& Future Work} \label{sec:conclusion}

In this manuscript, we analyze optimal control of an epidemic model which accounts for policy-makers enacting NPIs, but some portion of the population refusing to comply with the NPIs. We prove existence and uniqueness of a solution to our system of reaction-diffusion PDEs. We then prove existence of optimal control plans and derive a Lagrangian-based first order stationary system. Finally, we simulate the model under many different parameter regimes which represent policy-makers with differing agendas and demonstrate the interesting behavior that this model exhibits with respect to optimal control plans.

We conclude by suggesting a few avenues of future work. First, in this work, we are mainly concerned with non-pharmaceutical intervention (NPI), and neglect to consider prevention measures such as vaccination. Including vaccination or other such ``one-shot" type controls, and seeing how this affects the use of NPIs within our model could prove very interesting. In another direction, similar work to this (without the optimal control aspects) has been carried out in a network theoretic setting \cite{peng}, and a study of the synthesis between our work and the network theoretic work may shed new light on the manner in which diseases spread when the spread is affected by coevolving spread of opinions. \CB{A third avenue forward would be to study a system like ours, but wherein compliance also spreads through social contagion. Mathematically, this would add further competing nonlinearities to \eqref{PDE_SYS}, which could complicate the analysis, but given the self-reinforcing nature of social learning, this would open the door to interesting practical questions regarding natural segregation of compliant and noncompliant populations, and if such segregation occurs, one could explore the role that governmental controls may have in guiding the different populations to different areas.}  Next, one could imagine extending this to an $N$-player differential game where $N$ distinct populations are all experiencing the pandemic while also interacting in various ways. For example, this could account for the interplay between the optimal strategies of neighboring countries, and could include both interior control as in our model, and boundary control to account for closing borders. Alternately viewing this as either a cooperative control problem or a decentralized control problem, one could superficially assess the level to which international cooperation is necessary in addressing a pandemic. \CB{Another direction of future inquiry is the incorporation of human behavioral effects into multi-group models like those in \cite{MG1,MG2,MG3,MG4,MG5} or hybrid discrete-continuous models like those in \cite{HybridModel1,HybridModel2,HybridModel3}. This has the potential to significantly improve their modeling fidelity.} Finally, all of our work to this point is abstract and qualitative. Incorporating data and estimating parameters would be a crucial step in validating the model and pushing toward real world utility. \\

\section*{Acknowledgments} CP was partially supported by NSF-DMS grant 1937229 through the Data Driven Discovery Research Training Group at the University of Arizona. WW was partially supported by an AMS-Simons travel grant.\\

\bibliographystyle{abbrvurl} 
\bibliography{ref.bib}

\end{document}